\documentclass[11pt]{amsart}

\usepackage[colorlinks=true, pdfstartview=FitV, linkcolor=blue, citecolor=blue, urlcolor=blue, breaklinks=true]{hyperref}
\usepackage{amsmath,amsfonts,amssymb,amsthm,amscd,comment,paralist,etoolbox,mathtools,enumitem,mathdots,stmaryrd}
\usepackage[usenames,dvipsnames]{xcolor}
\usepackage{mathrsfs} 

\usepackage[all]{xy}
\usepackage{tikz}
\usetikzlibrary{arrows}
\usetikzlibrary{decorations.markings}

\newcommand{\pushright}[1]{\ifmeasuring@#1\else\omit\hfill$\displaystyle#1$\fi\ignorespaces}
\newcommand{\pushleft}[1]{\ifmeasuring@#1\else\omit$\displaystyle#1$\hfill\fi\ignorespaces}
\makeatother

\newcommand\C{\mathbb{C}}
\newcommand\Z{\mathbb{Z}}

\newcommand\R{\mathbb{R}}
\newcommand\N{\mathbb{N}}
\newcommand\F{\mathbb{F}}

\newcommand\kk{\Bbbk}
\newcommand\blambda{{\boldsymbol{\lambda}}}
\newcommand\bmu{{\boldsymbol{\mu}}}

\newcommand\fh{\mathfrak{h}}

\newcommand\cF{\mathcal{F}}

\newcommand\partition{\mathcal{P}}
\newcommand\cH{\mathcal{H}}
\newcommand\cB{\mathcal{B}}

\newcommand\bF{\mathbf{F}}

\newcommand\bc{\mathbf{c}}

\newcommand\one{\mathbf{1}}

\newcommand\Sy{\mathrm{Sym}}

\newcommand\op{\mathrm{op}}

\newcommand\gr{\mathrm{gr}}

\newcommand\tr{\mathrm{tr}}

\newcommand\sQ{\mathsf{Q}}

\newcommand\rot{\textup{rot}}   
\newcommand{\md}{\textup{-mod}}
\newcommand{\pmd}{\textup{-pmod}}

\newcommand\chern{\xi_B}        

\newcommand\ts{\textstyle}

\newcommand\even{\textup{even}}
\newcommand\odd{\textup{odd}}

\newcommand\redcircle[1]{\filldraw[fill=white, draw=red] #1 circle (3pt)}
\newcommand\bluedot[1]{\filldraw[blue] #1 circle (2pt)}

\tikzset{anchorbase/.style={>=stealth,baseline={([yshift=-0.5ex]current bounding box.center)}}}

\tikzset{->-/.style={decoration={
  markings,
  mark=at position #1 with {\arrow{>}}},postaction={decorate}}}
\tikzset{-<-/.style={decoration={
  markings,
  mark=at position #1 with {\arrow{<}}},postaction={decorate}}}

\DeclareMathOperator{\asc}{asc}

\DeclareMathOperator{\End}{End}
\DeclareMathOperator{\END}{END}
\DeclareMathOperator{\ev}{ev}
\DeclareMathOperator{\grdim}{grdim} 
\DeclareMathOperator{\Hom}{Hom}
\DeclareMathOperator{\HOM}{HOM}
\DeclareMathOperator{\id}{id}

\DeclareMathOperator{\Ob}{Ob}       

\DeclareMathOperator{\Span}{Span}

\DeclareMathOperator{\str}{str}
\DeclareMathOperator{\Sym}{Sym}
\DeclareMathOperator{\Tr}{Tr}       

\newtheorem{theo}{Theorem}[section]
\newtheorem{prop}[theo]{Proposition}
\newtheorem{lem}[theo]{Lemma}
\newtheorem{cor}[theo]{Corollary}

\theoremstyle{definition}

\newtheorem{rem}[theo]{Remark}
\newtheorem{eg}[theo]{Example}

\numberwithin{equation}{section}
\allowdisplaybreaks

\newtoggle{comments}
\newtoggle{details}
\newtoggle{prelimnote}
\newtoggle{detailsnote}

\toggletrue{comments}   

\toggletrue{detailsnote}   

\iftoggle{comments}{%
  \newcommand{\alcomments}[1]{
    \ \\
    {\color{red}
      \textbf{AL:} #1
    }
    \\
    }
  \newcommand{\drcomments}[1]{
    \ \\
    {\color{red}
      \textbf{DR:} #1
    }
    \\
    }
  \newcommand{\ascomments}[1]{
    \ \\
    {\color{red}
      \textbf{AS:} #1
    }
    \\
    }
}{%
  \newcommand{\alcomments}[1]{}
  \newcommand{\drcomments}[1]{}
  \newcommand{\ascomments}[1]{}
}

\iftoggle{details}{%
  \newcommand{\details}[1]{
      \ \\
      {\color{OliveGreen}
        \textbf{Details:} #1
      }
      \\
  }
}{%
  \newcommand{\details}[1]{}
}

\iftoggle{prelimnote}{%
  \newcommand{\prelim}{\textsc{Preliminary version} \bigskip}
}{%
  \newcommand{\prelim}{}
}

\begin{document}
%

\title{A graphical calculus for the Jack inner product on symmetric functions}

\author{Anthony Licata}
\address{A.~Licata: Mathematical Sciences Institute, Australian National University}
\urladdr{\url{http://maths-people.anu.edu.au/~licatat/Home.html}}
\email{amlicata@gmail.com}
\thanks{The first author was supported by a he first author was supported by Discovery Project grant
DP140103821 from the Australian Research Council.}

\author{Daniele Rosso}
\address{D.~Rosso: Department of Mathematics, University of California Riverside}
\urladdr{\url{http://pages.iu.edu/~drosso/}}
\email{drosso@iu.edu}

\author{Alistair Savage}
\address{A.~Savage: Department of Mathematics and Statistics, University of Ottawa}
\urladdr{\url{http://alistairsavage.ca}}
\email{alistair.savage@uottawa.ca}
\thanks{The third author was supported by Discovery Grant RGPIN-2017-03854 from the Natural Sciences and Engineering Research Council of Canada.}

\begin{abstract}
  Starting from a graded Frobenius superalgebra $B$, we consider a graphical calculus of $B$-decorated string diagrams.  From this calculus we produce algebras consisting of closed planar diagrams and of closed annular diagrams.  The action of annular diagrams on planar diagrams can be used to make clockwise (or counterclockwise) annular diagrams into an inner product space.  Our main theorem identifies this space with the space of symmetric functions equipped with the Jack inner product at Jack parameter $\dim B_\even - \dim B_\odd$.  In this way, we obtain a graphical realization of that inner product space.
\end{abstract}


\prelim

\maketitle
\thispagestyle{empty}

\tableofcontents

%
\section{Introduction}
%

Let $B$ be a nonnegatively graded Frobenius superalgebra over an algebraically closed field $\F$ of characteristic 0 (for example, the cohomology over $\F$ of a compact connected manifold).  Inspired by constructions of Khovanov in \cite{Kho14} and Cautis and the first author in \cite{CL12}, the second and third author, in \cite{RS15}, associated to $B$ a graded pivotal monoidal category $\cH_B^*$.  The objects of $\cH_B^*$ are formal direct sums of compact oriented 0-manifolds, and the morphisms are linear combinations of immersed oriented planar 1-manifolds, decorated by elements of the Frobenius algebra $B$, and subject to certain local relations.  Associated to $\cH_B^*$ are at least two potentially interesting algebraic objects:
\begin{itemize}
  \item the center $Z(\cH_B^*)$, which is the endomorphism algebra of the monoidal unit of the category $\cH_B^*$, is a graded supercommutative infinite-dimensional algebra, whose elements are linear combinations of immersed oriented closed 1-manifolds, decorated by elements of $B$, and subject to the local relations of the graphical calculus of $\cH_B^*$;
  \item the trace, or zeroth Hochschild homology, $\Tr(\cH_B^*)$ of $\cH_B^*$ is a graded noncommutative  infinite-dimensional algebra, whose elements are linear combinations of immersed closed annular 1-manifolds, decorated by elements of $B$, and subject to the local relations of the graphical calculus.
\end{itemize}
(We refer to Sections~\ref{sec:Heis-cat} and ~\ref{sec:trace} for the precise definitions of the monoidal category $\cH^*_B$ and the algebras $Z(\cH^*_B)$ and $\Tr(\cH^*_B)$.)  The algebra of annular diagrams $\Tr(\cH^*_B)$ acts linearly on the space of planar diagrams $Z(\cH^*_B)$.

The importance of the trace in diagrammatic categorification is first emphasized in the work of Beliakova--Habiro--Lauda with both Webster \cite{BHLW14} and with Guliyev \cite{BGHL14}.  In particular, the fact that the trace acts on the center of a pivotal monoidal category goes back at least to \cite{BGHL14}.  Nevertheless, at present not much is known about the algebra $\Tr(\cH^*_B)$, except in the case $B = \F$, where it was computed in \cite{CLLS15}.  The authors there suggest that, in the general case, $\Tr(\cH^*_B)$ should be understood in relation to the vertex algebra associated to the lattice $K_0(B)$, equipped with its Euler form.  We do not take up the computation of $\Tr(\cH^*_B)$ here.  However, there is a subcategory $\cH_B$ of the monoidal category $\cH^*_B$, defined by considering only those morphisms of degree zero, and the Hochschild homology $\Tr(\cH_B)$ is naturally a subalgebra of $\Tr(\cH^*_B)$ (see Proposition~\ref{prop:trace-injection}).  It turns out that there is a close relationship between $\Tr(\cH_B)$ and the ring of symmetric functions, and the goal of the present paper is to explain this relationship.

In order to explain the main result of the paper, we assume for simplicity here that the degree zero subalgebra of the graded Frobenius superalgebra $B$ is one-dimensional, that $B$ is not concentrated in degree zero, and that the Nakayama automorphism of $B$ is trivial. (The constructions in the body of the paper are written in greater generality than this.)  Our first main result, Theorem~\ref{theo:trace-identification}, states that the algebra of annular diagrams $\Tr(\cH_B)$ is isomorphic to a Heisenberg algebra $\fh_B$.  As a consequence, we obtain canonical isomorphisms between the subalgebras $\Tr(\cH_B)^+$ and $\Tr(\cH_B)^-$ of clockwise and counterclockwise annular diagrams and the algebra $\Sym$ of symmetric functions:
\begin{equation} \label{eq:intro-trpm-Sym-isom}
  \Tr(\cH_B)^\pm \cong \Sym.
\end{equation}

We then use the action of $\Tr(\cH_B)$ on $Z(\cH_B^*)$ to define a pairing
\[
	\langle -, - \rangle_B \colon \Tr(\cH_B)^+ \times \Tr(\cH_B)^- \longrightarrow \F.
\]
This pairing may be defined in two equivalent ways: graphically using the annular/planar diagrammatic realisations of
$\Tr(\cH_B)$ and $Z(\cH^*_B)$, or categorically, using the identifications of $\Tr(\cH^*_B)$ and $Z(\cH^*_B)$ as the Hochschild homology and center of the Heisenberg category $\cH^*_B$.  Our second main result,
Theorem~\ref{theo:Jack-diagrammatic}, identifies the pairing $\langle -, - \rangle_B$ with a bilinear form on $\Sym$ under the isomorphisms~\eqref{eq:intro-trpm-Sym-isom}.  More precisely, we prove that, after identifying $\Tr(\cH_B)^+$ and $\Tr(\cH_B)^-$ with $\Sym$, the pairing $\langle -, - \rangle_B$ is identified with the Jack pairing on $\Sym$.  That is, we have a commutative diagram
\[
  \xymatrix{
    \Tr(\cH_B)^- \times \Tr(\cH_B)^+ \ar[d]_{\cong} \ar[rr]^(0.65){\langle -, - \rangle_B} & & \F \\
    \Sym \times \Sym \ar[urr]_{\langle -, - \rangle_k} & &
  }
\]
where $\langle -, - \rangle_k \colon \Sym\times \Sym \to \F$ is the Jack pairing at Jack parameter $k = \dim B_\even - \dim B_\odd$.  (Here $B_\even$ and $B_\odd$ are the even and odd parts of $B$.)  Thus, we obtain a purely graphical realization of the algebra of symmetric functions with its Jack inner product at parameter $k = \dim B_\even - \dim B_\odd$.  In this sense, the Jack parameter $k$ is categorified by the graded Frobenius superalgebra $B$.

A special case, which is of some interest, is when the Frobenius algebra $B \cong \F[x]/(x^k)$ for a positive integer $k$.  The relationship between this Frobenius algebra and symmetric functions was investigated in \cite{CL16}, which used wreath products of $\F[x]/(x^k)$ with symmetric groups to give a categorical construction of a specialization of the Macdonald inner product (see \eqref{eq:specialized-Macdonald-pairing}).  The constructions of the current paper, on the other hand, give a parallel categorical construction of the Jack inner product, which is a further specialization of the inner products studied in \cite{CL16}.  Moreover, except for the Frobenius algebra $B$ itself, the basic input for the current paper is entirely graphical, involving only algebras of diagrams in the plane and the annulus.  So while the proof of Theorem~\ref{theo:Jack-diagrammatic} makes use of Heisenberg categories and their representation theory, the statement of Theorem~\ref{theo:Jack-diagrammatic} is purely a theorem in algebraic combinatorics.

This paper is organized as follows.  In Section~\ref{sec:lattice-Heis} we review the definition of lattice Heisenberg algebras and we recall the Heisenberg algebra $\fh_B$ associated to a graded Frobenius superalgebra $B$.  In Section~\ref{sec:Heis-cat} we recall the categories $\cH_B$ of \cite{RS15}, under some simplifying assumptions on $B$.  Then, in Section~\ref{sec:trace} we prove that $\Tr(\cH_B)$ injects into $\Tr(\cH_B^*)$ (after collapsing the grading and parity shifts), compute $\Tr(\cH_B)$, and identify the action of $\Tr(\cH_B)$ on the center $Z(\cH_B^*)$ with the Fock space representation of $\fh_B$.  In Section~\ref{sec:diagrammatic-pairing} we define our diagrammatic pairing and show that it categorifies the Jack bilinear form on symmetric functions.  In Section~\ref{sec:filtrations} we define a filtration on $Z(\cH_B^*)$ and prove some results about the associated graded algebra.  In particular, we relate the action of $\Tr(\cH_B^*)$ on $Z(\cH_B^*)$ to the multiplication in $Z(\cH_B^*)$.   In Section~\ref{sec:further-directions} we discuss some natural further directions of research suggested by the current work.

In Appendix~\ref{appendix}, we show, in a completely general setting, how one can obtain a number of different presentations of lattice Heisenberg algebras that appear naturally from categorification.  We also include some material related to lattice Heisenberg algebras arising from the Macdonald inner product on symmetric functions.  We deduce presentations of these algebras which may be of independent interest, and we explain how the limiting procedure that produces the Jack inner product from the Macdonald inner product can be interpreted in terms of incorporating a differential at the categorical level.

\iftoggle{detailsnote}{
\subsection*{Note on the arXiv version} For the interested reader, the tex file of the \href{https://arxiv.org/abs/1610.01862}{arXiv version} of this paper includes hidden details of some straightforward computations and arguments that are omitted in the pdf file.  These details can be displayed by switching the \texttt{details} toggle to true in the tex file and recompiling.
}{}

\subsection*{Acknowledgements}
The authors would like to thank Sabin Cautis, Josh Sussan, Aaron Lauda, and Peter Samuelson for a number of helpful conversations regarding traces in Heisenberg categorification.  The third author would also like to thank Shun-Jen Cheng for helpful conversations about associative superalgebras.

%
\section{Lattice Heisenberg algebras} \label{sec:lattice-Heis}
%

\subsection{Definitions} \label{subsec:lattice-Heis-def}

Throughout this paper, $\F$ will denote an algebraically closed field of characteristic zero.  Consider the rings
\begin{equation}
  \Z_{q,\pi} = \Z[q,q^{-1},\pi]/(\pi^2-1)
  \quad \text{and} \quad
  \F_{q,\pi} = \F[q,q^{-1},\pi]/(\pi^2-1).
\end{equation}

Let $L$ be a free $\Z_{q,\pi}$-module on the set $\{v_i\}_{i \in I}$, equipped with a nondegenerate symmetric sesquilinear form
\[
  \langle -, - \rangle_L \colon L \times L \to \Z_{q,\pi}.
\]
Here we say a form is \emph{sesquilinear} if it is $\Z$-bilinear, and
\[
  \langle v, q^s \pi^\epsilon w \rangle_L
  = q^s \pi^\epsilon \langle v, w \rangle_L
  = \langle q^{-s} \pi^\epsilon v, w \rangle_L,\quad
  \text{for all } v,w \in L,\ s \in \Z,\ \epsilon \in \Z_2.
\]
For $i,j \in I$, we will write $\langle i,j \rangle$ for $\langle v_i, v_j \rangle$.

We let $\Sym$ denote the Hopf algebra of symmetric functions with coefficients in $\F_{q,\pi}$.  If $p_n$ denotes the $n$-th power sum, then $\Sym$ has a basis given by
\[
  p_\lambda = p_{\lambda_1} p_{\lambda_2},\dotsc, p_{\lambda_\ell},\quad \lambda = (\lambda_1,\dotsc,\lambda_\ell) \in \partition,
\]
where $\partition$ denotes the set of partitions.  We let $|\lambda| = \sum_{1=1}^\ell \lambda_\ell$ denote the size of a partition $\lambda = (\lambda_1,\dotsc,\lambda_\ell) \in \partition$.  We will assume some familiarity with basic properties of symmetric functions.  We refer the reader to \cite{Mac95} for an exposition of this topic.

Let
\[ \ts
  \fh_L^+ = \fh_L^- = \Sym \otimes_{\Z_{q,\pi}} L \cong \bigoplus_{i \in I} \Sym.
\]
We will add a superscript $+$ or $-$ to various symmetric functions (or elements of $\Sym \otimes_{\Z_{q,\pi}} L$) and generating functions to indicate that they are to be considered as elements of $\fh_L^+$ or $\fh_L^-$, respectively.  In addition, for $f \in \Sym$, we write $f_i$, $i \in I$, for $f \otimes v_i \in \Sym \otimes_{\Z_{q,\pi}} L$.  For example $p_{n,i}^+$, $n \in \N$, $i \in I$, is $p_n \otimes v_i \in \fh_L^+$.

For $n \in \Z$, consider the $\F$-algebra homomorphism
\[
  \theta_n \colon \F_{q,\pi} \to \F_{q,\pi},\quad \theta_n(q) = q^n,\quad \theta_n(\pi) = -(-\pi)^n.
\]
There is a unique sesquilinear Hopf pairing determined by
\begin{equation} \label{eq:pn-pairing-determined-by-L}
  \langle -, - \rangle \colon \fh_L^- \times \fh_L^+ \to \F_{q,\pi},\quad
  \langle p_{n,i}^-, p_{m,j}^+ \rangle = \delta_{n,m} n \theta_n \big( \langle i, j \rangle \big),\quad
  n,m \in \N_+,\ i,j \in I.
\end{equation}

\begin{rem} \label{rem:rank-one-pairing}
  When $I = \{i\}$, with $\langle i, i \rangle = d \in \F_{q,\pi}$, the pairing \eqref{eq:pn-pairing-determined-by-L} is given on the basis $p_\lambda$, $\lambda \in \partition$, (here we drop the index $i$) by
  \begin{equation} \label{eq:full-pairing-determined-by-V}
    \langle p_\lambda^-, p_\mu^+ \rangle
    = \delta_{\lambda, \mu} z_\lambda \prod_{k=1}^{\ell(\lambda)} \theta_{\lambda_k} (d),\quad
    z_\lambda = \prod_{k \ge 1} k^{m_k(\lambda)} m_k(\lambda)!,
  \end{equation}
  where $m_k(\lambda)$ is the number of parts of $\lambda$ equal to $k$ and $\ell(\lambda)$ is the number of nonzero parts.  (The general formula follows from the values of the pairing on the power sums as in \cite[Example~(b), p.~306]{Mac95}.)  One can write such an explicit formula in the more general setting of arbitrary $I$, but the expression is more complicated (see, for example, \cite[Prop.~3.1]{Ber15} and \cite[(8.1)]{RS15b}).
\end{rem}

Let $\fh_L = \fh_L^+ \# \fh_L^-$ be the Heisenberg double associated to the Hopf pairing \eqref{eq:pn-pairing-determined-by-L}.  Note that, since the pairing is sesquilinear (as opposed to $\F_{q,\pi}$-bilinear, as in \cite{SY15,RS15c,RS15b,RS15}), we have $\F_{q,\pi}$-vector space isomorphisms
\[
  \fh_L \cong \fh_L^+ \otimes_{\F_{q,\pi}} \fh_L^- \cong \fh_L^- \otimes_{\F_{q,\pi}} \fh_L^+,
\]
where the action of $\F_{q,\pi}$ on $\fh_L^-$ is twisted via the $\F$-linear involution of $\F_{q,\pi}$ determined by $q \mapsto q^{-1}$, $\pi \mapsto \pi$.

The Heisenberg double $\fh_L$ is generated as an $\F_{q,\pi}$-algebra by $p_{n,i}^\pm$, $n \in \N_+$, $i \in I$, with relations
\begin{equation}
  [p_{n,i}^+, p_{m,j}^+] = 0,\quad [p_{n,i}^-,p_{m,j}^-] = 0,\quad [p_{n,i}^+, p_{m,j}^-] = \delta_{n,m} n \theta_n(\langle i,j \rangle),\quad n,m \in \N_+.
\end{equation}
Note that $\fh_L$ is naturally a graded algebra, where for a degree $n$ element $f \in \Sym$ and $i \in I$, we declare $f_i^\pm$ to be of degree $\pm n$.  In particular, $p_{n,i}^\pm$ is of degree $\pm n$.

\subsection{Heisenberg algebras associated to a Frobenius algebra} \label{subsec:hB}

We now relate the lattice Heisenberg algebras described above to the Heisenberg algebras associated to a graded Frobenius superalgebra in \cite[\S5]{RS15}.

Let $B = \bigoplus_{n \in \N} B_n$ be an $\N$-graded Frobenius superalgebra over $\F$ with trace map $\tr_B \colon B \to \F$ of $\Z$-degree $-\delta$, $\delta \in \N$.  (In other words, $\delta$ is the top degree of $B$.)  For simplicity, we assume that the trace map of $B$ is supersymmetric and even.  In particular, $\tr_B(b_1 b_2) = (-1)^{\bar b_1 \bar b_2} \tr_B(b_2 b_1)$ for all homogeneous $b_1, b_2 \in B$.  Here $\bar b$ denotes the parity of a homogeneous element $b \in B$.  In addition, we let $|b|$ denote the $\Z$-degree of a homogeneous element $b$.  Whenever we write an expression involving parities and/or degree, we shall implicitly assume that the corresponding element is homogeneous.  In the language of \cite{RS15}, our assumptions mean that $\psi = \id_B$ and $\delta=0$.  (The $\delta$ of \cite{RS15} is not the same as the $\delta$ of the current paper.)  We shall also assume that
\begin{center}
  \emph{all simple $B$-modules are of type $M$ (i.e.\ not isomorphic to their parity shifts).}
\end{center}
Therefore, in the language of \cite{RS15}, we have $\kk = \Z_{q,\pi}$.

We can naturally associate a lattice to $B$ as follows.  We let $L = L_B =K_0(B)$ be the split Grothendieck group of the category $B\pmd$ of finitely-generated projective graded $B$-modules, with coefficients in $\F$.  More precisely, $K_0(B)$ is the quotient of the free $\F$-module on isomorphism classes of finitely-generated projective graded $B$-modules, by the $\F$-submodule generated by $[M_1]_{\cong} - [M_2]_{\cong} + [M_3]_{\cong}$ for all $M_1,M_2,M_3 \in B\pmd$ such that $M_2 \cong M_1 \oplus M_3$.  Here $[M]_{\cong}$ denotes the isomorphism class of $M \in B\pmd$.

We define the structure of an $\F_{q,\pi}$-module on $K_0(B)$ in the usual way, defining
\[
  q^s \pi^\epsilon [M]_{\cong} = [\{s,\epsilon\}M],
\]
where $\{s,\epsilon\}M$ denotes the shift of $M$ determined by
\[
  (\{s,\epsilon\}M)_{s',\epsilon'} = M_{s'-s,\epsilon'+\epsilon}.
\]
Then $K_0(B)$ has a basis over $\F_{q,\pi}$ given by the classes of the indecomposable projective modules $P_i$, $i \in I$, of $B$.  We adopt the convention that the $P_i$ are generated in degree $(0,0)$.  We then define a sesquilinear pairing $\langle -,- \rangle \colon L \times L \to \F_{q,\pi}$ determined by
\begin{equation} \label{eq:K0-HOM-pairing}
  \langle [M], [N] \rangle = \grdim \HOM_B (M,N).
\end{equation}

Let $\fh_B$ be the Heisenberg algebra associated to $B$ in \cite[Def.~5.1]{RS15}, but where we extend scalars to $\F_{q,\pi}$ (i.e.\ $\fh_B \otimes_{\Z_{q,\pi}} \F_{q,\pi}$ in the language of \cite{RS15}).

\begin{prop} \label{prop:hB=hV}
  We have that $\fh_B$ is isomorphic, as a graded $\F_{q,\pi}$-algebra, to the algebra $\fh_L$ defined in Section~\ref{subsec:lattice-Heis-def} with $L=L_B$ the lattice associated to $B$ as above.
\end{prop}

\begin{proof}
  The proposition follows immediately from a comparison of the presentations of $\fh_B$ given in \cite[Prop.~5.5]{RS15} to the presentations of $\fh_L$ given in Propositions~\ref{prop:hh-presentation} and~\ref{prop:he-presentation}.
\end{proof}

In the sequel, we will identify $\fh_B$ and $\fh_L$ via the isomorphism of Proposition~\ref{prop:hB=hV}.

\subsection{Recovering the pairing}

The algebra $\fh_L$ acts naturally on $\fh_L^+$ via the Fock space representation (see, for example, \cite[\S2]{RS15c} or \cite[Def.~2.10]{SY15}).  Define the $\F_{q,\pi}$-linear map
\begin{equation} \label{eq:pi0-def}
  \kappa_0 = \langle 1_{\fh_L^-}, - \rangle \colon \fh_L^+ \to \F_{q,\pi}.
\end{equation}
In other words, $\kappa_0$ is projection onto the degree zero piece of $\fh_L^+$, where the grading on $\fh_L^+$ is induced by the natural grading on $\Sym$ (see Section~\ref{subsec:lattice-Heis-def}).  Equivalently, $\kappa_0$ is the counit of the Hopf algebra $\fh_L^+$.  The following lemma illustrates how one can recover the pairing \eqref{eq:pn-pairing-determined-by-L} from the algebra $\fh_L$.

\begin{lem} \label{lem:recover-pairing}
  We have
  \[
    \kappa_0 \left( (fg) \cdot 1_{H^+} \right) = \langle f,g \rangle,\quad \text{for all } f \in \fh_L^-,\ g \in \fh_L^+.
  \]
\end{lem}

\begin{proof}
  We have
  \[
    \kappa_0 \left( (fg) \cdot 1_{H^+} \right)
    = \kappa_0 ( f \cdot g )
    = \langle 1_{\fh_L^-}, f \cdot g, \rangle
    = \langle f, g \rangle,
  \]
  where, in the third equality, we have used the fact that the Fock space action of $\fh_L^-$ on $\fh_L^+$ is adjoint to multiplication in $\fh_L^-$.
\end{proof}

\subsection{The Jack pairing} \label{subsec:Jack-pairing}

Fix a $\Z$-graded super vector space $V$.  Then we can consider the setup of Section~\ref{subsec:lattice-Heis-def} in the special case that $L$ is a lattice of rank one, generated by an element $v$ satisfying $\langle v, v \rangle = \grdim V$.

Specializing even further, fix $k \in \N$ and define $V = \C[x]/(x^k)$.  If we declare $\deg x = (1,0) \in \Z \times \Z_2$, then
\[
  \grdim V = 1 + q + q^2 + \dotsb + q^{k-1} = \frac{1-q^k}{1-q},
\]
and the pairing~\eqref{eq:full-pairing-determined-by-V} becomes
\begin{equation} \label{eq:specialized-Macdonald-pairing}
  \langle p_\lambda^-, p_\mu^+ \rangle
  = \delta_{\lambda,\mu} z_\lambda \prod_{i=1}^{\ell(\lambda)} \frac{1-q^{k \lambda_i}}{1-q^{\lambda_i}},\quad \lambda,\mu \in \partition.
\end{equation}
This is a specialization of the pairing defining the Macdonald symmetric functions (see \eqref{eq:Macdonald-pairing}).  Note that the quotients appearing in \eqref{eq:specialized-Macdonald-pairing} are, in fact, polynomials in $q$.  Therefore, we can specialize further to $q=1$.  This yields the \emph{Jack pairing} at parameter $k$, given by
\begin{equation} \label{eq:Jack-pairing}
  \langle p_\lambda^-, p_\mu^+ \rangle_k = \delta_{\lambda,\mu} k z_\lambda, \quad \lambda, \mu \in \partition,
\end{equation}
and used to define the Jack symmetric functions (see \cite[\S VI.10]{Mac95}).

\begin{rem}
  The choice of $V = \F[x]/(x^k)$, and its relevance for the categorification of the ring of symmetric functions with a specialization of the Macdonald bilinear form, is discussed in \cite{CL16}.  See also Section~\ref{subsec:Jack-limit-differential} for an explanation of how we can view the above construction as arising from one related to the Macdonald inner product.
\end{rem}

%
\section{The graphical calculus of the Heisenberg category} \label{sec:Heis-cat}
%

\subsection{The Heisenberg category} \label{subsec:Heis-cat-def}

As in Section~\ref{subsec:hB}, we let $B = \bigoplus_{n \in \N} B_n$ be an $\N$-graded Frobenius superalgebra over $\F$ with even, supersymmetric trace map $\tr_B \colon B \to \F$ of $\Z$-degree $-\delta$, $\delta \in \N$.  In this section we recall the definition, given in \cite{RS15}, of the Heisenberg category $\cH_B$ associated to $B$.  (In fact, the construction in \cite{RS15} is more general, working for an arbitrary $\N$-graded Frobenius superalgebra, without the simplifying assumptions on $B$ we have made here.)

Given a basis $\cB$ of $B$, we let $\cB^\vee = \{b^\vee \mid b \in \cB\}$ denote the right dual basis defined by the property that
\[
  \tr_B(b_1 b_2^\vee) = \delta_{b_1,b_2} \quad \text{for all } b_1,b_2 \in \cB.
\]

We define $\cH_B'$ to be an $\F$-linear strict monoidal category whose objects are generated by symbols $\{n,\epsilon\} \sQ_+$ and $\{n,\epsilon\} \sQ_-$, for $n\in \Z$, $\epsilon \in \Z_2$.  We think of $\{n,\epsilon\}\sQ_+$ as being a shifted version of $\sQ_+$ and we declare the monoidal structure to be compatible with shifts $\{ \cdot, \cdot \}$, so that, for example, $\{s,\epsilon\} \sQ_-\otimes \{s', \epsilon'\} \sQ_- = \{s+s',\epsilon+\epsilon'\} (\sQ_-\otimes \sQ_-)$.  We will usually omit the $\otimes$ symbol, and write tensor products as words in $\sQ_+$ and $\sQ_-$.  Thus an arbitrary object of $\cH'_B$ is a finite direct sum of words in the letters $\sQ_+$ and $\sQ_-$ where each word has a shift.

The space of morphisms between two objects is the $\F$-algebra generated by suitable planar diagrams.  The diagrams consist of oriented compact one-manifolds immersed into the plane strip $\R \times [0,1]$ modulo certain local relations.  The grading on morphisms, which will be specified later in this section, determines the difference in shifts between the domain and codomain.

A single upward oriented strand denotes the identity morphism from $\sQ_+$ to $\sQ_+$ while a downward oriented strand denotes the identity morphism from $\sQ_-$ to $\sQ_-$.
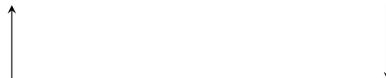
\[
  \begin{tikzpicture}[anchorbase]
    \draw[->] (0,0) -- (0,1);
    \draw[<-] (5,0) -- (5,1);
  \end{tikzpicture}
\]

Strands are allowed to carry dots labeled by elements of $B$. For example, if $b, b', b'' \in B$, then the diagram
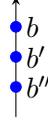
\[
  \begin{tikzpicture}[anchorbase]
    \draw [->](0,0) -- (0,1.6);
    \filldraw [blue](0,.4) circle (2pt);
    \draw (0,.4) node [anchor=west] [black] {$b''$};
    \filldraw [blue](0,.8) circle (2pt);
    \draw (0,.8) node [anchor=west] [black] {$b'$};
    \filldraw [blue](0,1.2) circle (2pt);
    \draw (0,1.2) node [anchor=west] [black] {$b$};
  \end{tikzpicture}
\]
is an element of $\Hom_{\cH'_B}(\sQ_+,\{-|b|-|b'|-|b''|, \bar b + \bar b' + \bar b''\}\sQ_+)$.  (See below for an explanation of the degree shift.)  Diagrams are linear in the dots in the sense that
\[
  \begin{tikzpicture}[anchorbase]
    \draw [->](0,0) -- (0,1);
    \bluedot{(0,.5)} node [anchor=west,color=black] {$(z_1b_1 + z_2 b_2)$};
  \end{tikzpicture}
  \ =\ z_1 \left(\
  \begin{tikzpicture}[anchorbase]
    \draw [->](0,0) -- (0,1);
    \bluedot{(0,.5)} node [anchor=west, color=black] {$b_1$};
  \end{tikzpicture} \right)
  \ +\ z_2 \left(\
  \begin{tikzpicture}[anchorbase]
    \draw [->](0,0) -- (0,1);
    \bluedot{(0,.5)} node [anchor=west, color=black] {$b_2$};
  \end{tikzpicture} \right)
  \quad \text{for } z_1,z_2 \in \F,\ b_1,b_2 \in B.
\]

Collision of dots is controlled by multiplication in the algebra $B$:

\noindent\begin{minipage}{0.5\linewidth}
  \begin{equation} \label{eq:collision-up}
    \begin{tikzpicture}[anchorbase]
      \draw[->] (0,0) -- (0,1.5);
      \bluedot{(0,.5)};
      \draw (0,.5) node [anchor=west] [black] {$b'$};
      \bluedot{(0,1)};
      \draw (0,1) node [anchor=west] [black] {$b$};
    \end{tikzpicture}
    = (-1)^{\bar b \bar b'}\
    \begin{tikzpicture}[anchorbase]
      \draw[->] (0,0) -- (0,1.5);
      \bluedot{(0,.75)};
      \draw (0,.75) node [anchor=west,color=black] {$b'b$};
    \end{tikzpicture}
  \end{equation}
\end{minipage}%
\begin{minipage}{0.5\linewidth}
  \begin{equation} \label{eq:collision-down}
    \begin{tikzpicture}[anchorbase]
      \draw[->] (0,1.5) -- (0,0);
      \bluedot{(0,.5)};
      \draw (0,.5) node [anchor=west] [black] {$b'$};
      \bluedot{(0,1)};
      \draw (0,1) node [anchor=west] [black] {$b$};
    \end{tikzpicture}
    \ =\
    \begin{tikzpicture}[anchorbase]
      \draw[->] (0,1.5) -- (0,0);
      \bluedot{(0,.75)};
      \draw (0,.75) node [anchor=west,color=black] {$bb'$};
    \end{tikzpicture}
  \end{equation}
\end{minipage}\par\vspace{\belowdisplayskip}
\noindent The sign in the first relation comes from the fact that composition of dots on upward strands actually corresponds to multiplication in the opposite algebra $B^\op$.  We allow dots to slide freely along strands.

We also allow strands to cross. For example
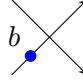
\[
  \begin{tikzpicture}[>=stealth]
    \draw[->] (0,0) -- (1,1);
    \draw[<-] (1,0) -- (0,1);
    \filldraw [blue](.25,.25) circle (2pt);
    \draw (.25,.25) node [anchor=south east] [black] {$b$};
  \end{tikzpicture}
\]
is an element of $\Hom_{\cH'_B}(\sQ_+\sQ_-, \{-|b|, \bar b\} \sQ_-\sQ_+)$.  (See below for an explanation of the degree shift.)  Notice that the domain of a morphism is specified at the bottom of the diagram and the codomain is specified at the top, so we compose diagrams by stacking them and reading upwards.

We assign a $\Z\times \Z_2$-grading to the space of planar diagrams by defining

\noindent\begin{minipage}{0.5\linewidth}
  \begin{equation}
    \deg\
    \begin{tikzpicture}[anchorbase]
      \draw [->](0,0) -- (1,1);
      \draw [->](1,0) -- (0,1);
    \end{tikzpicture}
    \ = (0, 0),
  \end{equation}
\end{minipage}%
\begin{minipage}{0.5\linewidth}
  \begin{equation}
    \deg\
    \begin{tikzpicture}[anchorbase]
      \draw [->](0,0) -- (0,1);
      \bluedot{(0,0.5)} node[anchor=west, color=black] {$b$};
    \end{tikzpicture}
    \ = \deg b,
  \end{equation}
\end{minipage}\par\vspace{\belowdisplayskip}

\noindent\begin{minipage}{0.5\linewidth}
  \begin{equation} \label{eq:rightcup-deg}
    \deg\
    \begin{tikzpicture}[anchorbase]
      \draw[->] (0,0) arc (180:360:.5);
    \end{tikzpicture}
    \ = (0,0),
  \end{equation}
\end{minipage}%
\begin{minipage}{0.5\linewidth}
  \begin{equation} \label{eq:rightcap-deg}
    \deg\
    \begin{tikzpicture}[anchorbase]
      \draw[<-] (6.5,-.5) arc (0:180:.5);
    \end{tikzpicture}
    \ = (0,0),
  \end{equation}
\end{minipage}\par\vspace{\belowdisplayskip}

\noindent\begin{minipage}{0.5\linewidth}
  \begin{equation} \label{eq:leftcup-deg}
    \deg\
    \begin{tikzpicture}[anchorbase]
      \draw[<-] (0,0) arc (180:360:.5);
    \end{tikzpicture}
    \ =(\delta,0),
  \end{equation}
\end{minipage}%
\begin{minipage}{0.5\linewidth}
  \begin{equation} \label{eq:leftcap-deg}
    \deg\
    \begin{tikzpicture}[anchorbase]
      \draw[->] (6.5,-.5) arc (0:180:.5);
    \end{tikzpicture}
    \ = (-\delta,0).
  \end{equation}
\end{minipage}\par\vspace{\belowdisplayskip}

We consider diagrams up to super isotopy fixing the vertical coordinates of the endpoints of strands.  By \emph{super isotopy}, we mean that dots on distinct strands supercommute when they move past each other vertically:
\begin{equation} \label{eq:supercomm}
  \begin{tikzpicture}[anchorbase]
    \draw[->] (0,0) -- (0,1.5);
    \bluedot{(0,.5)} node [anchor=east,color=black] {$b$};
    \draw (.5,.75) node {$\ldots$};
    \draw[->] (1,0)--(1,1.5);
    \bluedot{(1,1)} node [anchor=west,color=black] {$b'$};
  \end{tikzpicture}
  = \ (-1)^{\bar{b}\bar{b'}}\
  \begin{tikzpicture}[anchorbase]
    \draw[->] (6,0) --(6,1.5);
    \bluedot{(6,1)} node [anchor=east, color=black] {$b$};
    \draw (6.5,.75) node {$\ldots$};
    \bluedot{(7,.5)} node [anchor=west,color=black] {$b'$};
    \draw[->] (7,0) -- (7,1.5);
  \end{tikzpicture}.
\end{equation}
Because of the sign changes involved in super isotopy invariance, we do not allow diagrams in which odd dots appear at the same height (i.e.\ have the same vertical coordinate).

The local relations we impose are the following:

\noindent\begin{minipage}{0.5\linewidth}
  \begin{equation}\label{rel:dotslide-up-right}
    \begin{tikzpicture}[anchorbase]
      \draw[->] (0,0) -- (1,1);
      \draw[->] (1,0) -- (0,1);
      \bluedot{(.25,.25)} node [anchor=south east, color=black] {$b$};
    \end{tikzpicture}
    \ =\
    \begin{tikzpicture}[anchorbase]
      \draw[->](2,0) -- (3,1);
      \draw[->](3,0) -- (2,1);
      \bluedot{(2.75,.75)} node [anchor=north west, color=black] {$b$};
    \end{tikzpicture}
  \end{equation}
\end{minipage}%
\begin{minipage}{0.5\linewidth}
  \begin{equation}\label{rel:dotslide-up-left}
    \begin{tikzpicture}[anchorbase]
      \draw[->] (0,0) -- (1,1);
      \draw[->] (1,0) -- (0,1);
      \bluedot{(.75,.25)} node [anchor=south west, color=black] {$b$};
    \end{tikzpicture}
    \ =\
    \begin{tikzpicture}[anchorbase]
      \draw[->] (2,0) -- (3,1);
      \draw[->] (3,0) -- (2,1);
      \bluedot{(2.25,.75)} node [anchor=north east, color=black] {$b$};
    \end{tikzpicture}
  \end{equation}
\end{minipage}\par\vspace{\belowdisplayskip}

\noindent\begin{minipage}{0.5\linewidth}
  \begin{equation}\label{rel:up-braid}
    \begin{tikzpicture}[anchorbase]
      \draw (0,0) -- (2,2)[->];
      \draw (2,0) -- (0,2)[->];
      \draw[->] (1,0) .. controls (0,1) .. (1,2);
    \end{tikzpicture}
    \ =\
    \begin{tikzpicture}[anchorbase]
      \draw (3,0) -- (5,2)[->];
      \draw (5,0) -- (3,2)[->];
      \draw[->] (4,0) .. controls (5,1) .. (4,2);
    \end{tikzpicture}
  \end{equation}
\end{minipage}%
\begin{minipage}{0.5\linewidth}
  \begin{equation}\label{rel:up-up-double-cross}
    \begin{tikzpicture}[anchorbase]
      \draw[->] (0,0) .. controls (1,1) .. (0,2);
      \draw[->] (1,0) .. controls (0,1) .. (1,2);
    \end{tikzpicture}
    \ =\
    \begin{tikzpicture}[anchorbase]
      \draw[->] (2,0) --(2,2);
      \draw[->] (3,0) -- (3,2);
    \end{tikzpicture}
  \end{equation}
\end{minipage}\par\vspace{\belowdisplayskip}

\noindent\begin{minipage}{0.5\linewidth}
  \begin{equation}\label{rel:down-up-double-cross}
    \begin{tikzpicture}[anchorbase]
      \draw[<-] (0,0) .. controls (1,1) .. (0,2);
      \draw[->] (1,0) .. controls (0,1) .. (1,2);
    \end{tikzpicture}
    \ =\
    \begin{tikzpicture}[anchorbase]
      \draw[<-] (2,0) --(2,2);
      \draw[->] (3,0) -- (3,2);
    \end{tikzpicture}
    \ - \sum_{b \in \cB} \
    \begin{tikzpicture}[anchorbase]
      \draw (4,1.75) arc (180:360:.5);
      \draw (4,2) -- (4,1.75);
      \draw[<-] (5,2) -- (5,1.75);
      \draw (5,.25) arc (0:180:.5);
      \bluedot{(5,1.66)} node [anchor=west,color=black] {$b^\vee$};
      \bluedot{(4,.33)} node [anchor=west,color=black] {$b$};
      \draw (5,0) -- (5,.25);
      \draw[<-] (4,0) -- (4,.25);
    \end{tikzpicture}
  \end{equation}
\end{minipage}%
\begin{minipage}{0.5\linewidth}
  \begin{equation}\label{rel-up-down-double-cross}
    \begin{tikzpicture}[anchorbase]
      \draw[->] (0,0) .. controls (1,1) .. (0,2);
      \draw[<-] (1,0) .. controls (0,1) .. (1,2);
    \end{tikzpicture}
    \ =\
    \begin{tikzpicture}[anchorbase]
      \draw[->] (2,0) --(2,2);
      \draw[<-] (3,0) -- (3,2);
    \end{tikzpicture}
  \end{equation}
\end{minipage}\par\vspace{\belowdisplayskip}

\noindent\begin{minipage}{0.5\linewidth}
  \begin{equation}\label{rel:ccc}
    \begin{tikzpicture}[anchorbase]
      \draw [->](0,0) arc (180:360:0.5cm);
      \draw (1,0) arc (0:180:0.5cm);
      \bluedot{(0,0)} node [anchor=east,color=black] {$b$};
    \end{tikzpicture}
    \ = \tr_B(b)
  \end{equation}
\end{minipage}%
\begin{minipage}{0.5\linewidth}
  \begin{equation}\label{rel:left-curl}
    \begin{tikzpicture}[anchorbase]
      \draw (0,0) .. controls (0,.5) and (.7,.5) .. (.9,0);
      \draw (0,0) .. controls (0,-.5) and (.7,-.5) .. (.9,0);
      \draw (1,-1) .. controls (1,-.5) .. (.9,0);
      \draw[->] (.9,0) .. controls (1,.5) .. (1,1);
      \draw (1.5,0) node {$=$};
      \draw (2,0) node {$0$};
    \end{tikzpicture}
  \end{equation}
\end{minipage}\par\vspace{\belowdisplayskip}
\noindent Relation \eqref{rel:down-up-double-cross} is independent of the choice of basis ${\cB}$ by \cite[Lem~3.1]{RS15}.  By \emph{local relation}, we mean that relations \eqref{rel:dotslide-up-right}--\eqref{rel:left-curl} can be applied to any subdiagram contained in a horizontal strip of the form $\R \times [a,b]$, with $0 < a < b < 1$, where the remainder of the strip (to the right and left of the subdiagram) consists only of strands without dots.  If we wish to perform a local relation on a subdiagram in a vertical strip containing dots, we should first super isotope the diagram so that the vertical strip containing the subdiagram no longer contains dots.

The monoidal structure on morphisms is given by horizontal juxtaposition of diagrams.  However, because of \eqref{eq:supercomm}, it is important that we specify that, when juxtaposing diagrams, dots in the left-hand diagram should be placed higher (vertically) than dots in the right-hand diagram.  We always super isotope the diagrams before juxtaposing them to ensure that this is the case.

Note that all of the graphical relations are homogeneous.   Thus, the morphism spaces of $\cH_B'$ are $\Z \times\Z_2$-graded vector spaces and composition of morphisms is compatible with the grading.  We denote by $\{ \cdot,\cdot\}$ the grading shift in $\cH'_B$.  It will be important in the sequel to emphasize that, in the category $\cH_B'$, \emph{we only allow morphisms of degree zero}.  In particular, a linear combination of diagrams of degree $(-s,\epsilon)$ whose tops and bottoms correspond to the sequences $x$ and $y$, respectively, of the generating objects $\sQ_+$ and $\sQ_-$ is to be viewed as a degree \emph{zero} morphism from $x$ to $\{s,\epsilon\}y$.

We define the category $\cH_B$ to be the idempotent completion (also known as the Karoubi envelope) of $\cH'_B$.  By definition, the objects of $\cH_B$ consist of pairs $(x,e)$, where $x$ is an object of $\cH'_B$ and $e \colon x \to x$ is an idempotent morphism.  The space of morphisms in $\cH_B$ from $(x,e)$ to $(x',e')$ is the subspace of morphisms in $\cH'_B$ consisting of $g \colon x \to x'$ such that $ge = g = e'g$.  The idempotent $e$ defines the identity morphism of $(x,e)$ in $\cH_B$.

Note that $\cH_B$ inherits the $\Z\times \Z_2$-grading from $\cH'_B$.  This means that the split Grothendieck group $K_0(\cH_B)$ of $\cH_B$ (with coefficients in $\F$) is an $\F_{q,\pi}$-algebra.  Precisely, we define
\[
  q^s \pi^\epsilon [x]_{\cong} := [\{s, \epsilon\} x]_{\cong},\quad
  x \in \Ob \cH_B,\ s \in \Z,\ \epsilon \in \Z_2,
\]
where $[x]_{\cong}$ denotes the isomorphism class of $x$ (and its image in the Grothendieck group).  The monoidal structure in $\cH_B$ descends to a multiplication in $K_0(\cH_B)$.

By relaxing the condition that morphisms must be of degree zero, we obtain the graded category $\cH_B^*$ with the same objects as $\cH_B$, but with morphism spaces between objects $x$ and $y$ given by
\begin{equation} \label{eq:HB-graded-def}
  \Hom_{\cH_B^*}(x,y) = \HOM_{\cH_B}(x,y) := \bigoplus_{(s,\epsilon) \in \Z \times \Z_2} \Hom_{\cH_B}(x,\{s,\epsilon\}y).
\end{equation}

\subsection{The algebras of annular and closed diagrams} \label{sec:diagrammatic-algebras}

Associated to the graded category $\cH_B^*$ are two natural algebras.  The first is the space of annular diagrams, with multiplication given by nesting diagrams; so for annular diagrams $A_1$ and $A_2$, we define $A_1 A_2$ to be the annular diagram obtained by placing $A_2$ in the center region of $A_1$.  As for the monoidal structure in $\cH_B^*$, we always first super isotope the diagrams so that all the dots in $A_1$ are higher (vertically) than those in $A_2$.

The second algebra associated to $\cH_B^*$ is the space of closed diagrams, with multiplication given by vertical stacking; so for closed diagrams $D_1$ and $D_2$, we define $D_1 D_2$ to be the closed diagram obtained by placing $D_1$ above $D_2$.

The algebra of annular diagrams acts naturally on the space of closed diagrams by placing the closed diagram in the center of the annulus and viewing the result as a closed diagram.  See Figures~\ref{fig:annular-and-closed} and~\ref{fig:annulus-action-on-center}.  Again, we always first super isotope so that the dots in the annular diagram are higher than those in the closed diagram.
\begin{figure}
  \[
    A =\
    \begin{tikzpicture}[anchorbase]
      \filldraw[thick,draw=green!60!black,fill=green!20!white] (-1,-0.7) -- (-1,0.7) .. controls (-1,2) and (3,2) .. (3,0.7) -- (3,-0.7) .. controls (3,-2) and (-1,-2) .. (-1,-0.7);
      \filldraw[thick,draw=green!60!black,fill=white] (0.6,-0.25) -- (0.6,0.25) .. controls (0.6,0.5) and (1.4,0.5) .. (1.4,0.25) -- (1.4,-0.25) .. controls (1.4,-0.5) and (0.6,-0.5) .. (0.6,-.25);
      \draw[->-=0.23,color=black] (-0.4,0.5) arc(360:180:0.1) .. controls (-0.6,1.5) and (2.6,1.5) .. (2.6,0.5) -- (2.6,-0.5) .. controls (2.6,-1.5) and (-0.6,-1.5) .. (-0.6,-0.5) .. controls (-0.6,0) and (0,0) .. (0,0.5);
      \draw[-<-=0.23] (-0.4,0.5) .. controls (-0.4,1.3) and (2.4,1.3) .. (2.4,0.5) -- (2.4,-0.5) .. controls (2.4,-1.3) and (-0.4,-1.3) .. (-0.4,-0.5) arc(180:0:0.1);
      \draw[->-=0.23] (-0.2,0.5) .. controls (-0.2,1.1) and (2.2,1.1) .. (2.2,0.5) -- (2.2,-0.5) .. controls (2.2,-1.1) and (-0.2,-1.1) .. (-0.2,-0.5);
      \draw[->-=0.2] (0,0.5) .. controls (0,0.9) and (2,0.9) .. (2,0.5) -- (2,-0.5) .. controls (2,-0.9) and (0,-0.9) .. (0,-0.5) .. controls (0,0.2) and (0.2,-0.3) .. (0.2,0.4);
      \draw[->-=0.2] (0.2,0.4) .. controls (0.2,0.7) and (1.8,0.7) .. (1.8,0.4) -- (1.8,-0.4) .. controls (1.8,-0.7) and (0.2,-0.7) .. (0.2,-0.4) .. controls (0.2,0.1) and (-0.2,0) .. (-0.2,0.5);
      \bluedot{(-0.56,0.65)};
      \bluedot{(-0.2,0.5)};
      \bluedot{(-0.4,-0.1)};
      \redcircle{(0.2,0.3)};
    \end{tikzpicture}
    \qquad \qquad D=\
    \begin{tikzpicture}[anchorbase]
      \draw[thick,dashed,draw=green!60!black] (-1,-0.7) -- (-1,0.7) .. controls (-1,2) and (3,2) .. (3,0.7) -- (3,-0.7) .. controls (3,-2) and (-1,-2) .. (-1,-0.7);
      \draw[-<-=0.1,color=black] (1,-0.5) -- (1,0.5) .. controls (1,1.3) and (-0.7,1.3) .. (-0.7,0.5) -- (-0.7,-0.5) .. controls (-0.7,-1.3) and (1,-1.3) .. (1,-0.5);
      \draw[<-] (0.1,0.3) arc(0:360:0.3);
      \draw[->] (0.7,-0.3) arc(0:360:0.3);
      \draw[->] (1.5,0.5) arc (180:0:.5) .. controls (2.5,0) and (1.5,0) .. (1.5,-0.5) arc (180:360:0.5) .. controls (2.5,0) and (1.5,0) .. (1.5,0.5);
      \bluedot{(-0.5,0.3)};
      \redcircle{(0.1,-0.3)};
      \bluedot{(1.6,0.8)};
      \redcircle{(2,-1)};
    \end{tikzpicture}
  \]
  \caption{An annular diagram $A$ and a closed diagram $D$.  The closed dots are labeled by elements of $B$, and the open dots represent right curls (see~\eqref{eq:right-curl}).\label{fig:annular-and-closed}}
\end{figure}
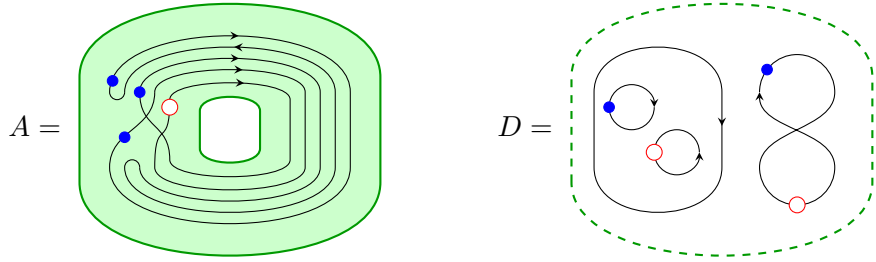

\begin{figure}
  \[
    A \cdot D =\
    \begin{tikzpicture}[anchorbase]
      \draw[thick,dashed,draw=green!60!black] (-1,-0.7) -- (-1,0.7) .. controls (-1,2) and (3,2) .. (3,0.7) -- (3,-0.7) .. controls (3,-2) and (-1,-2) .. (-1,-0.7);
      \draw[thick,dashed,draw=green!60!black] (0.6,-0.25) -- (0.6,0.25) .. controls (0.6,0.5) and (1.4,0.5) .. (1.4,0.25) -- (1.4,-0.25) .. controls (1.4,-0.5) and (0.6,-0.5) .. (0.6,-.25);
      \draw (1,0) node {$D$};
      \draw[->-=0.23,color=black] (-0.4,0.5) arc(360:180:0.1) .. controls (-0.6,1.5) and (2.6,1.5) .. (2.6,0.5) -- (2.6,-0.5) .. controls (2.6,-1.5) and (-0.6,-1.5) .. (-0.6,-0.5) .. controls (-0.6,0) and (0,0) .. (0,0.5);
      \draw[-<-=0.23] (-0.4,0.5) .. controls (-0.4,1.3) and (2.4,1.3) .. (2.4,0.5) -- (2.4,-0.5) .. controls (2.4,-1.3) and (-0.4,-1.3) .. (-0.4,-0.5) arc(180:0:0.1);
      \draw[->-=0.23] (-0.2,0.5) .. controls (-0.2,1.1) and (2.2,1.1) .. (2.2,0.5) -- (2.2,-0.5) .. controls (2.2,-1.1) and (-0.2,-1.1) .. (-0.2,-0.5);
      \draw[->-=0.2] (0,0.5) .. controls (0,0.9) and (2,0.9) .. (2,0.5) -- (2,-0.5) .. controls (2,-0.9) and (0,-0.9) .. (0,-0.5) .. controls (0,0.2) and (0.2,-0.3) .. (0.2,0.4);
      \draw[->-=0.2] (0.2,0.4) .. controls (0.2,0.7) and (1.8,0.7) .. (1.8,0.4) -- (1.8,-0.4) .. controls (1.8,-0.7) and (0.2,-0.7) .. (0.2,-0.4) .. controls (0.2,0.1) and (-0.2,0) .. (-0.2,0.5);
      \bluedot{(-0.56,0.65)};
      \bluedot{(-0.2,0.5)};
      \bluedot{(-0.4,-0.1)};
      \redcircle{(0.2,0.3)};
    \end{tikzpicture}
  \]
  \caption{The action of the annular diagram $A$ on the closed diagram $D$.\label{fig:annulus-action-on-center}}
\end{figure}
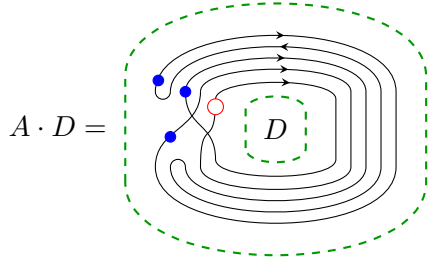

The goal of the current paper is to relate the two algebras mentioned above to the ring of symmetric functions.  In particular, we will identify certain annular diagrams with specific symmetric functions and use the action of annular diagrams on closed diagrams to obtain a diagrammatic realization of the Jack inner product.  Our computations will use the fact that the algebras of annular diagrams and closed diagrams correspond, respectively, to the trace and center of the category $\cH_B^*$.

%
\section{Trace decategorification} \label{sec:trace}
%

In this section we study the trace decategorification of the category $\cH_B$.  For a survey of trace decategorification that is well suited to the setup of the current paper, we refer the reader to \cite[\S3]{BGHL14}.

\subsection{Trace of the Heisenberg category} \label{subsec:trace-def}

By definition, the \emph{trace}, or \emph{zeroth Hochschild homology}, of $\cH_B'$ is given by
\begin{gather*}
  \Tr(\cH_B') = \left( \End \cH_B' \right) / \Span_\F \{fg-gf \mid f \in \Hom_{\cH_B'}(x,y),\ g \in \Hom_{\cH_B'}(y,x),\ x,y \in \Ob \cH_B'\}, \\
  \text{where} \quad \End \cH_B' = \bigoplus_{x \in \Ob \cH_B'} \Hom_{\cH_B'}(x,x).
\end{gather*}
We denote the image of a morphism $f \in \Hom_{\cH_B'}(x,x)$ in $\Tr(\cH_B')$ by $[f]$.  The trace $\Tr(\cH_B')$ is naturally a module over $\F_{q,\pi}$, where
\[
  q^s \pi^\epsilon [f] = [\{s,\epsilon\}f],\quad f \in \Hom_{\cH_B'}(x,x),\ s \in \Z,\ \epsilon \in \Z_2.
\]
(Note that $\{s,\epsilon\}f \in \Hom_{\cH_B'}(\{s,\epsilon\}x,\{s,\epsilon\}x)$.)  Since the trace remains unchanged under passage to the idempotent completion (see \cite[Prop.~3.2]{BGHL14}), we have $\Tr(\cH_B') = \Tr(\cH_B)$.

The monoidal structure on the category $\cH_B'$ induces a multiplication on the trace $\Tr(\cH_B')$.  Namely, given morphisms $f \colon x \to x'$ and $g \colon y \to y'$ in $\cH_B'$, then $fg \colon xy \to x'y'$, and we define $[f][g] = [fg]$.  With this multiplication, $\Tr(\cH_B')$ is an $\F_{q,\pi}$-algebra.

Recall that the objects of $\cH_B'$ are sums of (possibly empty) sequences $x$ of the generating objects $\sQ_+$ and $\sQ_-$.  We define the \emph{ascension} of such a sequence $x$, denoted $\asc x$, to be the total number of occurrences of $\sQ_+$ minus the total number of occurrences of $\sQ_-$.  Ascension then defines a natural $\Z$-grading on $\Tr(\cH_B')$ by declaring $\asc [f] = \asc x$ for $f \in \Hom_{\cH_B'}(x,x)$.

\begin{rem} \label{rem:morphism-grading}
  In the case that $B=\F$, the trace of $\cH_\F'$ was considered in \cite{CLLS15}.  However, in the case that the grading on $B$ is nontrivial, the trace of $\cH_B'$ behaves very differently.  This arises from the fact that the grading on $B$ induces a grading on the morphism spaces of $\cH_B'$ (see Section~\ref{subsec:Heis-cat-def}).  By definition, we require morphisms in $\cH_B'$ to have degree zero.  For example, the diagram
  \begin{equation} \label{eq:right-curl}
    \begin{tikzpicture}[anchorbase]
      \draw[->] (0,0) -- (0,1);
      \redcircle{(0,0.5)};
    \end{tikzpicture}
    \ :=\
    \begin{tikzpicture}[anchorbase,scale=0.5]
      \draw (2,1) .. controls (2,1.5) and (1.3,1.5) .. (1.1,1);
      \draw (2,1) .. controls (2,.5) and (1.3,.5) .. (1.1,1);
      \draw (1,0) .. controls (1,.5) .. (1.1,1);
      \draw[->] (1.1,1) .. controls (1,1.5) .. (1,2);
    \end{tikzpicture}
  \end{equation}
  has $\Z$-degree equal to the top degree $\delta$ of $B$.  In the category $\cH_B'$, we thus view this diagram as a degree zero morphism from $\sQ_+$ to $\{\delta,0\}\sQ_+$.  So, if $\delta \ne 0$, this is \emph{not} an endomorphism of any object in $\cH_B'$.  Therefore, we do not consider its class in the trace.  As we will see, this drastically simplifies the trace.  One should compare this to the fact that, for $\delta \ne 0$, one can show that the split Grothendieck group of $\cH_B'$ is isomorphic to the Heisenberg algebra $\fh_B$ (see \cite[Th.~10.5]{RS15} and \cite[Th.~1]{CL12}), whereas the analogous statement for $B = \F$ seems to be much harder to prove and is, in fact, still a conjecture (see \cite[Conj.~1]{Kho14}).
\end{rem}

The trace of the graded category $\cH_B^*$ is defined to be
\begin{gather*}
  \Tr(\cH_B^*) = \left( \END \cH_B \right) / \Span_\F \{fg-(-1)^{\bar f \bar g} gf \mid f \in \HOM_{\cH_B}(x,y),\ g \in \HOM_{\cH_B}(y,x),\ x,y \in \Ob \cH_B\}, \\
  \text{where} \quad \END \cH_B = \bigoplus_{x \in \Ob \cH_B} \HOM_{\cH_B}(x,x).
\end{gather*}
We again denote the image of a morphism $f \in \END \cH_B$ in $\Tr(\cH_B^*)$ by $[f]$, and we have $\Tr(\cH_B^*) = \Tr({\cH_B'}^*)$.  In $\Tr(\cH_B^*)$, we have
\begin{equation} \label{eq:q-pi-trace-action}
  q[f] = [f],\quad \pi [f] = -(-1)^{\bar f} [f],\quad f \in \END \cH_B.
\end{equation}
\details{
  Suppose $f \in \HOM_{\cH_B}(x,x)$ for some $x \in \Ob \cH_B$.  For $s \in \Z$ and $\epsilon \in \Z_2$, let $1_{s,\epsilon}$ denote the identity endomorphism of $x$, viewed as an element of $\Hom_{\cH_B}(x,\{s,\epsilon\}x)$.  Then we have
  \[
    q^s \pi^\epsilon [f]
    = [1_{s,\epsilon} f 1_{s,\epsilon}^{-1}]
    = (-1)^{\epsilon(\bar f + \epsilon)} [ 1_{s,\epsilon}^{-1} 1_{s,\epsilon} f]
    = (-1)^\epsilon (-1)^{\bar f \epsilon} [f].
  \]
}
Therefore, $\Tr(\cH_B^*)$ is naturally an $\F$-algebra.  Note that, in general, $\Tr(\cH_B^*)$ contains classes of many more endomorphisms compared to $\Tr(\cH_B)$.  For instance $\Tr(\cH_B^*)$ contains the class of the right curl \eqref{eq:right-curl}, since we now allow endomorphisms of nonzero degree.  See Remark~\ref{rem:W-alg-connection}.

By \cite[Lem.~10.1]{RS15}, every object of $\cH_B'$ is isomorphic to a direct sum of shifts of objects of the form $\sQ_+^n \sQ_-^m$, $n,m \in \N$.  Thus, by \cite[Lem~2.1]{BHLW14}, we have
\[
  \Tr(\cH_B) = \Tr(\cH_B') = \Tr(\cH_B''),
\]
where $\cH_B''$ is the full subcategory of $\cH_B'$ whose objects are shifts of $\sQ_+^n \sQ_-^m$, $n,m \in \N$.  Similarly,
\[
  \Tr(\cH_B^*) = \Tr({\cH_B'}^*) = \Tr({\cH_B''}^*),
\]
where ${\cH_B''}^*$ is the full subcategory of ${\cH_B'}^*$ whose objects are shifts of $\sQ_+^n \sQ_-^m$, $n,m \in \N$.

\subsection{An injection of traces}

Define
\[
  \END \cH_B'
  := \End {\cH_B'}^*
  = \bigoplus_{x \in \Ob \cH_B'} \Hom_{{\cH_B'}^*}(x,x)
  = \bigoplus_{x \in \Ob \cH_B'} \HOM_{\cH_B'}(x,x),
\]
and define $\END \cH_B''$ similarly.  Then we have a natural inclusion map
\begin{equation} \label{eq:End-into-END}
  \End \cH_B'
  \hookrightarrow \END \cH_B'.
\end{equation}
In the remainder of this paper, we will view $\F$ as an $\F_{q,\pi}$-module via $q \mapsto 1$, $\pi \mapsto -1$.

\begin{prop} \label{prop:trace-injection}
  Assume $B \ne B_0$ (i.e.\ $\delta > 0$).   The natural inclusion \eqref{eq:End-into-END} induces an injective $\F$-algebra homomorphism
  \begin{equation} \label{eq:trace-injection}
    \Tr(\cH_B) \otimes_{\F_{q,\pi}} \F \hookrightarrow \Tr(\cH_B^*).
  \end{equation}
\end{prop}

\begin{proof}
  Let
  \begin{gather*}
    K = \Span_\F \{fg-gf \mid f \in \Hom_{\cH_B'}(x,y),\ g \in \Hom_{\cH_B'}(y,x),\ x,y \in \Ob \cH_B''\} \subseteq \End \cH_B'',\\
    K^* = \Span_\F \{fg-(-1)^{\bar f \bar g} gf \mid f \in \HOM_{\cH_B'}(x,y),\ g \in \HOM_{\cH_B'}(y,x),\ x,y \in \Ob \cH_B''\} \subseteq \End {\cH_B''}^*.
  \end{gather*}
  Thus,
  \[
    \Tr(\cH_B) = (\End \cH_B'')/K
    \quad \text{and} \quad
    \Tr(\cH_B^*) = (\END \cH_B'')/K^*.
  \]
  We clearly have
  \[
    K \subseteq K^* \cap \End \cH_B'',
  \]
  and, by \eqref{eq:q-pi-trace-action}, we have an induced map $\Tr(\cH_B) \otimes_{\F_{q,\pi}} \F \to \Tr(\cH_B^*)$.

  It remains to prove that
  \[
    K^* \cap \End \cH_B'' \subseteq K.
  \]
  Suppose $f \in \HOM_{\cH_B'}(x,y)$, $g \in \HOM_{\cH_B'}(y,x)$ for some $x,y \in \Ob \cH_B''$.  By the definition of $\cH_B''$, we have $x = \sQ_+^{n_x} \sQ_-^{m_x}$ and $y = \sQ_+^{n_y} \sQ_-^{m_y}$ for some $n_x,m_x,n_y,m_y \in \N$.  Thus, the argument used in the proof of \cite[Lem.~10.2]{RS15} shows that $f$ and $g$ have no nonzero negative degree components.  Thus, if $fg-gf$ lies in degree zero, we must have $fg - gf = f_0 g_0 - g_0 f_0$, where $f_0$ and $g_0$ are the degree zero pieces of $f_0$ and $g_0$, respectively.  The result follows.
\end{proof}

\subsection{Diagrammatic realization of the trace}

In the case that $B=\F$, it was explained in \cite[\S4]{CLLS15} that the trace of $\cH_\F$ can be viewed as diagrams on the annulus.  We can do the same for the categories $\cH_B$ and $\cH_B^*$.  However, as noted in Remark~\ref{rem:morphism-grading}, we must be careful to only consider morphisms of degree zero.  For example, consider the relation \eqref{rel:down-up-double-cross}.  All diagrams in this equation are of degree zero, and thus correspond to elements of $\End_{\cH_B'}(\sQ_- \sQ_+)$.  However, in the trace, we will want to slide the cup or cap around the annulus.  In this case, we are considering the rightmost term of \eqref{rel:down-up-double-cross} as a composition of morphisms
\[
  \sQ_- \sQ_+ \to \{\delta + |b|\} \varnothing \to \sQ_- \sQ_+,
\]
and we must keep track of the grading shift on the object $\varnothing$.

Given $f \in \End \cH_B'$, we depict the element $[f] \in \Tr (\cH_B)$ by drawing $f$ on the left side of an annulus and then closing up the diagram to the right:
\begin{equation} \label{eq:f-on-annulus}
  \begin{tikzpicture}[anchorbase]
    \draw (-0.5,-0.25) -- (-0.5,0.25) -- (0.5,0.25) -- (0.5,-0.25) -- (-0.5,-0.25);
    \draw (0,0) node {$f$};
    \draw (-0.4,0.25) -- (-0.4,0.7);
    \draw (-0.2,0.25) -- (-0.2,0.7);
    \draw (0,0.25) -- (0,0.7);
    \draw (0.2,0.25) -- (0.2,0.7);
    \draw (0.4,0.25) -- (0.4,0.7);
    \draw (-0.4,-0.25) -- (-0.4,-0.7);
    \draw (-0.2,-0.25) -- (-0.2,-0.7);
    \draw (0,-0.25) -- (0,-0.7);
    \draw (0.2,-0.25) -- (0.2,-0.7);
    \draw (0.4,-0.25) -- (0.4,-0.7);
  \end{tikzpicture}
  \quad \rightsquigarrow \quad
  \begin{tikzpicture}[anchorbase]
    \filldraw[thick,draw=green!60!black,fill=green!20!white] (-1,-0.7) -- (-1,0.7) .. controls (-1,2) and (3,2) .. (3,0.7) -- (3,-0.7) .. controls (3,-2) and (-1,-2) .. (-1,-0.7);
    \filldraw[thick,draw=green!60!black,fill=white] (0.6,-0.25) -- (0.6,0.25) .. controls (0.6,0.5) and (1.4,0.5) .. (1.4,0.25) -- (1.4,-0.25) .. controls (1.4,-0.5) and (0.6,-0.5) .. (0.6,-.25);
    \filldraw[draw=black,fill=white] (-0.7,-0.25) -- (-0.7,0.25) -- (0.3,0.25) -- (0.3,-0.25) -- (-0.7,-0.25);
    \draw (-0.2,0) node {$f$};
    \draw (-0.6,0.25) -- (-0.6,0.5) .. controls (-0.6,1.5) and (2.6,1.5) .. (2.6,0.5) -- (2.6,-0.5) .. controls (2.6,-1.5) and (-0.6,-1.5) .. (-0.6,-0.5) -- (-0.6,-0.25);
    \draw (-0.4,0.25) -- (-0.4,0.5) .. controls (-0.4,1.3) and (2.4,1.3) .. (2.4,0.5) -- (2.4,-0.5) .. controls (2.4,-1.3) and (-0.4,-1.3) .. (-0.4,-0.5) -- (-0.4,-0.25);
    \draw (-0.2,0.25) -- (-0.2,0.5) .. controls (-0.2,1.1) and (2.2,1.1) .. (2.2,0.5) -- (2.2,-0.5) .. controls (2.2,-1.1) and (-0.2,-1.1) .. (-0.2,-0.5) -- (-0.2,-0.25);
    \draw (0,0.25) -- (0,0.5) .. controls (0,0.9) and (2,0.9) .. (2,0.5) -- (2,-0.5) .. controls (2,-0.9) and (0,-0.9) .. (0,-0.5) -- (0,-0.25);
    \draw (0.2,0.25) -- (0.2,0.4) .. controls (0.2,0.7) and (1.8,0.7) .. (1.8,0.4) -- (1.8,-0.4) .. controls (1.8,-0.7) and (0.2,-0.7) .. (0.2,-0.4) -- (0.2,-0.25);
  \end{tikzpicture}
\end{equation}
Elements of $\Tr(\cH_B^*)$ can be viewed similarly.

\begin{rem}
  Above, we have taken the \emph{right} trace, which corresponds diagrammatically to closing off morphisms of $\cH_B'$ or ${\cH_B'}^*$ to the right.  One can also take the \emph{left} trace by closing off to the left.  The relation between the two constructions is as follows.  The adjunctions (i.e.\ the cups and caps) give involutions of the categories $\cH_B$ and $\cH_B^*$ that interchange $\sQ_+$ and $\sQ_-$ and reverse the order of tensor products.  Diagrammatically, this automorphism $f \mapsto f^\vee$ comes from rotating diagrams through an angle of $\pi$.  Then, closing off $f$ to the right yields the same annular diagram as closing off $f^\vee$ to the left.  The resulting anti-automorphism from the left trace to the right trace interchanges $\Tr(\cH_B)^\pm$ (see \eqref{eq:trace-pm}) and corresponds to the anti-automorphism from the Heisenberg double $\fh_B^+ \# \fh_B^-$ to the opposite Heisenberg double $\fh_B^- \# \fh_B^+$.  Whereas the action of the left trace on the center, to be described in Section~\ref{subsec:trace-action-on-center}, corresponds to the lowest weight Fock space representation (generated by a vector killed by $\Tr(\cH_B)^-$), the natural action of the right trace on the center corresponds to the highest weight Fock space representation.
\end{rem}

\subsection{Identification of the trace}

For $n \in \N_+$, let
\[
  B^{\rtimes n} := B^{\otimes n} \rtimes S_n
\]
be the wreath product algebra associated to $B$, with grading inherited from $B$ (i.e.\ we take $S_n$ to lie in degree zero).  By convention, we set $B^{\rtimes 0} = \F$.  The degree zero piece of $B^{\rtimes n}$ is $B^{\rtimes n}_0 = B_0 \rtimes S_n$.

Let $e_1,\dotsc,e_\ell$, be idempotents in $B$ such that $B e_1,\dotsc,B e_\ell$, is a complete list of (pairwise non-isomorphic) representatives of the isomorphism classes of indecomposable projective $B$-modules, up to grading and parity shift.  Note that $e_i \in B_0$ for all $i \in \{1,\dotsc,\ell\}$.  Similarly, for $\lambda \vdash n$, we let $e_\lambda$ be an idempotent of $\F S_n$ such that $(\F S_n) e_\lambda$ is the simple module corresponding to the partition $\lambda$.

For $n \in \N$, define
\[ \ts
  \partition^\ell(n) = \left\{ \blambda = (\lambda^1,\dotsc,\lambda^\ell) \in \partition^\ell \mid |\blambda| := \sum_{i=1}^\ell |\lambda^i| = n \right\}.
\]
For $\blambda = (\lambda^1,\dotsc,\lambda^\ell) \in \partition^\ell(n)$, let
\[
  e_\blambda := \left( e_1^{\otimes |\lambda^1|} \otimes \dotsb \otimes e_\ell^{\otimes |\lambda^\ell|} \right) \left(e_{\lambda^1} \otimes \dotsb \otimes e_{\lambda^\ell} \right)
\]
be the corresponding idempotent of $B^{\rtimes n}$, where
\[
  e_{\lambda^1}
  \otimes\dotsb\otimes e_{\lambda^\ell} \in \mathbb{F} S_{|\lambda_1|} \otimes \dotsb \otimes \mathbb{F} S_{|\lambda^\ell|} \subseteq \mathbb{F} S_n.
\]
Then $B^{\rtimes n} e_\blambda$, $\blambda \in \partition^\ell(n)$, is a complete set of indecomposable projective $B^{\rtimes n}$-modules, up to isomorphism and shift.  (See, for example, \cite[Prop.~4.4]{RS15}.)

For $\lambda \vdash n$ and $i \in I$, we set
\[
  e_{\lambda,i} = e_i^{\otimes n} e_\lambda,\quad i \in \{1,\dotsc,\ell\},\ \lambda \vdash n.
\]

We have natural algebra homomorphisms
\begin{equation} \label{eq:FSn-end-spaces}
  (B^{\rtimes n})^\op \to \END_{\cH'_B}(\sQ_+^n),\quad B^{\rtimes n} \to \END_{\cH'_B}(\sQ_-^n)
\end{equation}
(see \cite[(9.1),(9.2)]{RS15}).  For an element of $B^{\rtimes n}$, we will use superscripts $+$ and $-$ to denote the image of this element in $\END_{\cH_B'}(\sQ_+^n)$ and $\END_{\cH_B'}(\sQ_-^n)$, respectively, under the above maps.  For instance, for $i \in I$ and $\lambda \vdash n$, $e_{\lambda,i}^+$ is the idempotent $e_{\lambda,i} \in B^{\rtimes n}$, viewed as an endomorphism of $\sQ_+^n$.

Recall that the \emph{trace}, or \emph{cocenter}, of a superalgebra $A$ is the vector space
\[
  \Tr(A) := A/ \Span_\F \{ ab - (-1)^{\bar a \bar b} ba \mid a,b \in A \}.
\]

\begin{lem} \label{lem:cocenter-basis}
  Suppose $A$ is a finite-dimensional semisimple superalgebra whose simple modules are all of type $M$.  Let $b_1,\dotsc,b_s$ be a set of idempotents such that $Ab_1,\dotsc,Ab_s$ is a complete set of (pairwise non-isomorphic) representatives of simple $A$-modules.  Then the images of $b_1,\dotsc,b_s$ form a basis of $\Tr(A)$.
\end{lem}

\begin{proof}
  The assumptions on $A$ imply that it is isomorphic to a product of superalgebras of the form $M_{m|n}$ for $m,n \in \N$, not both zero.  Thus, it suffices to consider the case $A = M_{m|n}$.  Then we have
  \[
    \Tr(A) = A / \Span_\F \{XY - (-1)^{\bar X \bar Y} YX \} = M_{m|n}/\{X \mid \str X = 0\},
  \]
  where $\str X$ denotes the supertrace of $X \in M_{m|n}$.  Now, up to conjugation by a degree zero element, we can take the minimal idempotent $b$ of $A$ to be the matrix with a 1 in the $(1,1)$-entry and zeros elsewhere.  Then $\str b = 1$ and the result follows.
\end{proof}

The \emph{Chern character} for $\cH_B$ is the homomorphism of $\F_{q,\pi}$-algebras
\[
  \chern \colon K_0(\cH_B) \to \Tr(\cH_B) = \Tr(\cH_B'),\quad [x]_{\cong} \mapsto [1_x],\quad x \in \Ob \cH_B.
\]

\begin{theo} \label{theo:trace-identification}
  Assume $B \ne B_0$ (i.e.\ $\delta > 0$).  If $B_0$ is semisimple, then the map
  \[
    \chern \colon K_0(\cH_B) \otimes_\Z \F \to \Tr (\cH_B)
  \]
  is an isomorphism of $\F_{q,\pi}$-algebras.  Furthermore, we have an isomorphism of $\F_{q,\pi}$-algebras uniquely determined by
  \begin{equation} \label{eq:trace-hB-isom}
    \varphi_B \colon \Tr(\cH_B) \cong \fh_B,\quad
    \left[ e_{\lambda,i}^\pm \right] \mapsto s_{\lambda,i}^\pm,
  \end{equation}
  where $s_\lambda$ denotes the Schur function associated to $\lambda \in \partition$.
\end{theo}

\begin{proof}
  Recall from Section~\ref{subsec:trace-def} that we have $\Tr (\cH_B') = \Tr (\cH_B'')$, where $\cH_B''$ is the full subcategory of $\cH_B'$ whose objects are shifts of $\sQ_+^n \sQ_-^m$, $n,m \in \N$.  By \cite[Lem.~10.2]{RS15}, we have
  \[
    \End_{\cH_B'}(\sQ_+^n \sQ_-^m) = \END_{\cH_B'}(\sQ_+^n \sQ_-^m)_{0,0} \cong (B^{\rtimes n}_0)^\op \otimes_\F B^{\rtimes m}_0,\quad n,m \in \N,
  \]
  where the isomorphism (which we will view as identification) comes from \eqref{eq:FSn-end-spaces}.  It follows that
  \[
    \Tr(\cH_B) = \bigoplus_{n,m \in \N} C_n^\op \otimes C_m,
  \]
  where, for $n \in \N$,
  \[
    C_n = B^{\rtimes n}_0 / \Span_\F \{fg-gf \mid f,g \in B^{\rtimes n}_0\}
  \]
  is the cocenter of the ring $B^{\rtimes n}_0$.

  By \cite[Th.~10.5]{RS15}, $K_0(\cH_B)$ has an $\F_{q,\pi}$-basis given by
  \[
    \left[ \left( \sQ_+^{|\blambda|} \sQ_-^{|\bmu|},e_\blambda^+ e_\bmu^- \right) \right]_{\cong},\quad
    \blambda, \bmu \in \partition^\ell.
  \]
  Since
  \[
    \chern \left( \left[ \left( \sQ_+^{|\blambda|} \sQ_-^{|\bmu|},e_\blambda^+ e_\bmu^- \right) \right]_{\cong} \right) = [e_\blambda^+ e_\bmu^-],
  \]
  it follows from Lemma~\ref{lem:cocenter-basis} that $\chern$ is an isomorphism.  (Note that injectivity also follows from \cite[Prop.~2.4]{BHLW14}.)

  The isomorphism $\varphi_B$ in the last statement of the theorem is given by the composition
  \[
  	\varphi_B = \alpha \circ \chern^{-1},
  \]
  where $\alpha \colon K_0(\cH_B) \to \fh_B$ is the isomorphism of $\F$-algebras from  \cite[Thm.~10.5]{RS15} (after tensoring over $\Z$ with $\F$).  To see that
  \[
  	\varphi_B \left( \left[ e_{\lambda,i}^\pm \right] \right) = s_{\lambda,i}^\pm,
  \]
  it suffices to show that
  \begin{equation}\label{eq:Schur-iso}
    \alpha \left(	\left[ \left( \sQ_\pm^n,e_{\lambda,i} \right) \right]_{\cong} \right)
    = s_{\lambda,i}^\pm
  \end{equation}
  for all partitions $\lambda \vdash n$.  By \cite[Prop.~5.5]{RS15}, \cite[Th.~9.2]{RS15}, and Proposition~\ref{prop:hh-presentation}, the equation \eqref{eq:Schur-iso} holds for a one-part partition $\lambda = (n)$ (and also for the transpose $\lambda = (1^n)$).  Now, straightforward generalizations of the proof of \cite[Lem.~5]{CL12} and \cite [Rem.~6]{CL12} show that the expression for $\left[ \left( \sQ_\pm^n,e_{\lambda,i} \right) \right]_{\cong}$ as a linear combination of products of the generators  $\left[ \left( \sQ_\pm^{m},e_{(m),i} \right) \right]_{\cong}$, $m \in \N$, is given by the Giambelli rules in the ring of symmetric functions for expressing the Schur functions in terms of complete homogeneous symmetric functions. On the other hand, in $\fh_B$, the expressions for $s_{\lambda,i}^{\pm}$ in terms of the generators $s_{(m),i}^{\pm}$ are also given by the Giambelli rules.   Thus, since \eqref{eq:Schur-iso} holds when $\lambda = (n)$, and $\alpha$ is an algebra isomorphism, it follows that \eqref{eq:Schur-iso} holds for all $\lambda$.
\end{proof}

\begin{rem} \label{rem:varphiB-graded}
  Recall that the algebra $\fh_B$ has a natural $\Z$-grading, where we view a degree $n$ symmetric function in $H^\pm$ as living in degree $\pm n$ (see Section~\ref{subsec:lattice-Heis-def}).  Then the isomorphism $\varphi_B$ sends elements in $\Tr(\cH_B)$ of ascension $n$ to degree $n$ elements of $\fh_B$, for $n \in \Z$.  Thus $\varphi_B$ is also an isomorphism of \emph{graded} algebras.
\end{rem}

\begin{rem} \label{rem:W-alg-connection}
  One should compare Theorem~\ref{theo:trace-identification} to \cite[Th.~1]{CLLS15}, which considered the case $B=\F$.  In that case, the trace $\Tr(\cH_\F)$ is a quotient of a W-algebra, and is much larger than the corresponding Heisenberg algebra $\fh_\F$.  This is because of the difference pointed out in Remark~\ref{rem:morphism-grading}.  Of course, instead of $\cH_B$ one could consider the corresponding graded category $\cH_B^*$.  As noted in the introduction, the trace of $\cH_B^*$ should be related to W-algebras.
\end{rem}

Let
\begin{equation} \label{eq:trace-pm}
  \Tr(\cH_B)^\pm = \bigoplus_{n \in \N} \left[ \End_{\cH_B'}(\sQ_\pm^n) \right].
\end{equation}
By Theorem~\ref{theo:trace-identification}, we have an isomorphism of $\F_{q,\pi}$-modules
\begin{equation} \label{eq:trace-pm-decomp}
  \Tr(\cH_B) \cong \Tr(\cH_B)^+ \otimes_{\F_{q,\pi}} \Tr(\cH_B)^-,
\end{equation}
and
\[
  \varphi_B \left( \Tr(\cH_B)^\pm \right) = \fh_B^\pm.
\]

For $n \in \N$ and $i \in \{1,\dotsc,\ell\}$, let $P_{n,i}^\pm$ be the class in $\Tr(\cH_B)$ of the element of $\End_{\cH_B'} (\sQ_\pm^n)$ corresponding, under the homomorphism \eqref{eq:FSn-end-spaces}, to $e_i^{\otimes n} w_n$, where $w_n$ is a cycle in $S_n$ of length $n$.  Then
\begin{equation} \label{eq:Plambda-def}
  P_{\lambda,i}^\pm := P_{\lambda_1,i}^\pm P_{\lambda_2,i}^\pm \dotsm P_{\lambda_\ell,i}^\pm,\quad \lambda = (\lambda_1,\dotsc,\lambda_\ell) \in \partition,
\end{equation}
is the class of the element of $\End_{\cH_B'} (\sQ_\pm^n)$ corresponding, under the homomorphism \eqref{eq:FSn-end-spaces}, to $e_i^{\otimes n} w_\lambda$, where $w_\lambda$ is a permutation in $S_n$ of cycle type $\lambda \vdash n$.  For example, diagrammatically,
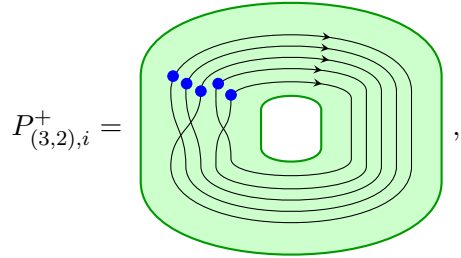
\[
  P_{(3,2),i}^+ =\
  \begin{tikzpicture}[anchorbase]
    \filldraw[thick,draw=green!60!black,fill=green!20!white] (-1,-0.7) -- (-1,0.7) .. controls (-1,2) and (3,2) .. (3,0.7) -- (3,-0.7) .. controls (3,-2) and (-1,-2) .. (-1,-0.7);
    \filldraw[thick,draw=green!60!black,fill=white] (0.6,-0.25) -- (0.6,0.25) .. controls (0.6,0.5) and (1.4,0.5) .. (1.4,0.25) -- (1.4,-0.25) .. controls (1.4,-0.5) and (0.6,-0.5) .. (0.6,-.25);
    \draw[->-=.25,color=black] (-0.6,0.5) .. controls (-0.6,1.5) and (2.6,1.5) .. (2.6,0.5) -- (2.6,-0.5) .. controls (2.6,-1.5) and (-0.6,-1.5) .. (-0.6,-0.5) .. controls (-0.6,0) and (-0.2,0) .. (-0.2,0.5);
    \draw[->-=.25,color=black] (-0.4,0.5) .. controls (-0.4,1.3) and (2.4,1.3) .. (2.4,0.5) -- (2.4,-0.5) .. controls (2.4,-1.3) and (-0.4,-1.3) .. (-0.4,-0.5) .. controls (-0.4,0) and (-0.6,0) .. (-0.6,0.5);
    \draw[->-=.25,color=black] (-0.2,0.5) .. controls (-0.2,1.1) and (2.2,1.1) .. (2.2,0.5) -- (2.2,-0.5) .. controls (2.2,-1.1) and (-0.2,-1.1) .. (-0.2,-0.5) .. controls (-0.2,0) and (-0.4,0) .. (-0.4,0.5);
    \draw[->-=.25,color=black] (0,0.4) -- (0,0.5) .. controls (0,0.9) and (2,0.9) .. (2,0.5) -- (2,-0.5) .. controls (2,-0.9) and (0,-0.9) .. (0,-0.5) -- (0,-0.4) .. controls (0,0.1) and (0.2,-0.1) .. (0.2,0.4);
    \draw[->-=.25,color=black] (0.2,0.4) .. controls (0.2,0.7) and (1.8,0.7) .. (1.8,0.4) -- (1.8,-0.4) .. controls (1.8,-0.7) and (0.2,-0.7) .. (0.2,-0.4) .. controls (0.2,0.1) and (0,-0.1) .. (0,0.4);
    \bluedot{(-0.57,0.7)};
    \bluedot{(-0.39,0.6)};
    \bluedot{(-0.2,0.5)};
    \bluedot{(0.03,0.6)};
    \bluedot{(0.2,0.45)};
  \end{tikzpicture}\ ,
\]
where the dots are labeled by the idempotent $e_i \in B$.

\begin{prop} \label{prop:Pp-isom}
  Recalling the isomorphism $\varphi_B$ of \eqref{eq:trace-hB-isom}, we have
  \[
    \varphi_B \left( P_{\lambda,i}^\pm \right) = p_{\lambda,i}^\pm,\quad \lambda \in \partition,\ i \in \{1,\dotsc,\ell\}.
  \]
\end{prop}

\begin{proof}
  We fix $i \in \{1,\dotsc,\ell\}$ and drop the subscript $i$ from the notation.  Let $L_\lambda$ denote the irreducible representation of the symmetric group $S_n$ indexed by the partition $\lambda\vdash n$, with character $\chi_\lambda$.  For $\mu \vdash n$, let
  \[
    C_\mu = \left( \sum_{g \text{ of cycle type }\mu} g \right) \in Z(\C[S_n])
  \]
  denote the sum over permutations of cycle type $\mu$.  The set $\{C_\mu\}_{\mu\vdash n}$ is a basis of the center $Z(\C[S_n])$.  Another basis of $Z(\C[S_n])$ is given by the primitive central idempotents $\{\tilde{e}_\lambda\}$.  We have
  \[
    \chi_\lambda(\tilde{e}_\mu) = \delta_{\mu,\lambda} \dim L_\lambda
    \quad \text{and} \quad
  	\chi_\lambda(C_\mu) = z_\mu \chi_\lambda(g_\mu),
  \]
  where $g_\mu$ is any permutation of cycle type $\mu$ and $z_\mu$ is the size of the conjugacy class of type $\mu$ (see \eqref{eq:full-pairing-determined-by-V}).  Thus, in $Z(\C[S_n])$, we may express the basis $\{C_\mu\}$ in terms of the basis $\{\tilde{e}_\lambda\}$ by
  \[
    C_\mu = \sum_{\lambda} \frac{z_{\mu}}{\dim L_\lambda} \chi_\lambda(g_\mu) \tilde{e}_\lambda.
  \]
  \details{
    We can write
    \[
      C_\mu = \sum_\lambda a_{\mu,\lambda} \tilde{e}_\lambda.
    \]
    Thus, for $\lambda \in \partition$, we have
    \[
      z_\mu \chi_\lambda(g_\mu)
      = \chi_\lambda(C_\mu)
      = \sum_\nu a_{\mu,\nu} \chi_\lambda(\tilde{e}_\nu)
      = a_{\mu,\lambda} \dim L_\lambda
      \implies a_{\mu,\lambda} = \frac{z_{\mu}}{\dim L_\lambda} \chi_\lambda(g_\mu).
    \]
  }
  Now, under the map $\iota \colon Z(\C[S_n]) \rightarrow \Tr(\C[S_n])$, we have
  \[
    \iota(C_\mu) = z_\mu \iota(g_\mu), \text{ and } \iota(\tilde{e}_\lambda) = (\dim L_\lambda) \iota(e_\lambda),
  \]
  where we recall that $e_\lambda \in \C[S_n]$ is a Young idempotent such that $\C[S_n]e_\lambda \cong L_\lambda$.  It follows that for any permutation $g_\mu\in S_n$ of cycle type $\mu$, we have
  \[
  	\iota(g_\mu) = \sum_{\lambda} \chi_\lambda(g_\mu) \iota(e_\lambda).
  \]
  In other words, in $\Tr(\C[S_n])$, the transition matrix between the basis given by the classes of permutations and the basis given by Young idempotents is given by the character table of $S_n$.  On the other hand, by \cite[(I.7.8)]{Mac95}, the transition matrix between the power sum basis and the Schur function basis is also given by the character table of $S_n$:
  \[
   	p_\mu = \sum_{\lambda} \chi_\lambda(g_\mu) s_\lambda.
  \]
  Thus, in any linear map (such as $\varphi_B$) from $\Tr(\C[S_n])$ to symmetric functions, if the class of the Young idempotent $e_\lambda$ is sent to the Schur function $s_\lambda$ for all $\lambda$, then the class of a permutation $g_\mu$ of cycle type $\mu$ will be sent to the power sum symmetric function $p_\mu$.  This completes the proof.
\end{proof}

\begin{rem}
  Proposition~\ref{prop:Pp-isom} illustrates an important feature of the trace categorification of $\fh_B$ given by Theorem~\ref{theo:trace-identification}.  Namely, the power sum symmetric functions have a natural interpretation as classes of certain permutations.  In the Grothendieck group categorification of $\fh_B$ given in \cite[Th.~10.5]{RS15}, the power sum symmetric functions do not seem to have such a natural interpretation.
\end{rem}

\subsection{Action of the trace on the center} \label{subsec:trace-action-on-center}

All objects of $\cH_B$ have biadjoints up to grading and parity shifts, and all adjunctions in $\cH_B$ are cyclic.  Thus, the trace $\Tr(\cH_B^*)$ acts on the graded center
\[
  Z(\cH_B^*) := \HOM_{\cH_B^*} (\varnothing,\varnothing),
\]
as we now explain.  Note that $Z(\cH_B^*)$ is naturally an $\F$-algebra, but does not have the natural structure of an $\F_{q,\pi}$-module arising from shifts, since the nontrivial shift of an endomorphism of $\varnothing$ is no longer an endomorphism of $\varnothing$.  (Rather, it is an endomorphism of a shift of $\varnothing$.)

Essentially, the action of $\Tr(\cH_B^*)$ on $Z(\cH_B^*)$ is the usual one for a pivotal monoidal category (or, more generally, a pivotal 2-category).  See, for example, \cite[\S9.1]{BHLW14}.  However, since in the category $\cH_B$ the objects $\sQ_+$ and $\sQ_-$ are only biadjoint up to a grading shift, we are forced to work in the graded category $\cH_B^*$ instead of $\cH_B$ itself.  In addition, we need to carefully define the grading conventions in our action.

The action is defined as follows.  The graded center $Z(\cH_B^*)$ consists of linear combinations of closed diagrams $D$, which are morphisms $\varnothing \to \{- \deg D\} \varnothing$, where $\deg D \in \Z \times \Z_2$ is the degree of $D$.  The action of $[f] \in \Tr(\cH_B^*)$ on such a diagram $D$ is given by placing $D$ in the interior of the annulus and interpreting the resulting diagram as an element in $Z(\cH_B^*)$.  (As explained in Section~\ref{sec:diagrammatic-algebras}, we always first isotope the diagrams so that any dots in the annular diagram are higher than any dots in the closed diagram.)  Thus, if
\[
  f \in \HOM_{\cH_B'} \left( \sQ_+^n \sQ_-^m, \sQ_+^n \sQ_-^m \right),
\]
is of degree $\deg f \in \Z \otimes \Z_2$, then
\[
  [f] \cdot D =
  \begin{tikzpicture}[anchorbase]
    \draw (-0.7,-0.25) -- (-0.7,0.25) -- (0.3,0.25) -- (0.3,-0.25) -- (-0.7,-0.25);
    \draw (-0.2,0) node {$f$};
    \draw (-0.6,0.25) -- (-0.6,0.5) .. controls (-0.6,1.5) and (2.6,1.5) .. (2.6,0.5) -- (2.6,-0.5) .. controls (2.6,-1.5) and (-0.6,-1.5) .. (-0.6,-0.5) -- (-0.6,-0.25);
    \draw (-0.4,0.25) -- (-0.4,0.5) .. controls (-0.4,1.3) and (2.4,1.3) .. (2.4,0.5) -- (2.4,-0.5) .. controls (2.4,-1.3) and (-0.4,-1.3) .. (-0.4,-0.5) -- (-0.4,-0.25);
    \draw (-0.2,0.25) -- (-0.2,0.5) .. controls (-0.2,1.1) and (2.2,1.1) .. (2.2,0.5) -- (2.2,-0.5) .. controls (2.2,-1.1) and (-0.2,-1.1) .. (-0.2,-0.5) -- (-0.2,-0.25);
    \draw (0,0.25) -- (0,0.5) .. controls (0,0.9) and (2,0.9) .. (2,0.5) -- (2,-0.5) .. controls (2,-0.9) and (0,-0.9) .. (0,-0.5) -- (0,-0.25);
    \draw (0.2,0.25) -- (0.2,0.4) .. controls (0.2,0.7) and (1.8,0.7) .. (1.8,0.4) -- (1.8,-0.4) .. controls (1.8,-0.7) and (0.2,-0.7) .. (0.2,-0.4) -- (0.2,-0.25);
    \draw (1,0) node {$D$};
  \end{tikzpicture}
  \in \HOM_{\cH_B'} \big( \{n \delta, 0\} \varnothing, \{(m \delta, 0) - \deg D - \deg f\} \varnothing \big) \subseteq Z(\cH_B^*).
\]
Note in particular that the degree of $[f]$ as an operator on $Z(\cH_B^*)$ is
\begin{equation} \label{eq:degree-of-operator-on-center}
  \deg f + (n \delta - m \delta, 0) = \deg f + (\delta \asc [f],0).
\end{equation}

Via the injection of Proposition~\ref{prop:trace-injection}, we have an induced action of $\Tr(\cH_B)$ on $Z(\cH_B^*)$.  The identity morphism $\one_\varnothing$ of the trivial object (empty sequence) of $\cH_B$ is the empty diagram.  Acting on the empty diagram gives a map
\[
  \Tr(\cH_B) \to Z(\cH_B^*),\quad [f] \mapsto [f] \cdot 1_\varnothing.
\]
We define
\begin{equation} \label{eq:FB-def}
  \cF_B = \Tr(\cH_B) \cdot 1_\varnothing.
\end{equation}
In light of Proposition~\ref{prop:Fock-space-isom} below, we call $\cF_B$ the \emph{diagrammatic Fock space}.  Let $F_B = \fh_B^+$ denote the Fock space representation of $\fh_B$, and let $F_{B,\F} = F_B \otimes_{\F_{q,\pi}} \F$.

\begin{prop} \label{prop:Fock-space-isom}
  Assume $B \ne B_0$ (i.e.\ $\delta > 0$).  We have an isomorphism of $\F$-modules
  \begin{equation} \label{eq:Fock-space-isom}
    \phi_B \colon \cF_B \to F_B \otimes_{\F_{q,\pi}} \F,\quad [e_\blambda^+] \cdot 1_\varnothing \mapsto s_\blambda^+ := s_{\lambda^1,1}^+ \dotsm s_{\lambda^\ell,\ell}^+,\quad \blambda = (\lambda^1,\dotsc,\lambda^\ell) \in \partition^\ell,
  \end{equation}
  extended by linearity.  Furthermore, the diagram
  \[
    \xymatrix{
      \Tr(\cH_B) \otimes \cF_B \ar[r] \ar[d]_{\varphi_B \otimes \phi_B} & \cF_B \ar[d]^{\phi_B}_{\cong} \\
      \fh_B \otimes F_{B,\F} \ar[r] & F_{B,\F}
    }
  \]
  commutes, where the horizontal arrows are given by the Fock space actions.
\end{prop}

\begin{proof}
  As in the proof of Theorem~\ref{theo:trace-identification}, $\Tr(\cH_B)$ is spanned by the classes of $(B^{\rtimes n}_0)^\op \otimes_\F B^{\rtimes m}_0 \cong \End_{\cH_B'}(\sQ_+^n \sQ_-^m)$, $n,m \in \N$.  Now, classes of elements of $\End_{\cH_B'} \sQ_-^m$ have negative ascensions, and thus are operators on $Z(\cH_B^*)$ of strictly negative degree.  However, by \cite[Prop.~8.9]{RS15}, $Z(\cH_B^*)$ is concentrated in nonnegative degree.  It follows that
  \begin{equation} \label{eq:tr-action-on-empty-diag}
    \cF_B
    = \Tr(\cH_B)^+ \cdot 1_\varnothing
  \end{equation}

  Now, $\cF_B \ne 0$ by \cite[Th.~7.4]{RS15}.  (For example, a clockwise circle is mapped to $n$ by the functor $\bF_n$ of \cite[Th.~7.4]{RS15}, and thus cannot be zero.)  Therefore, by the Stone--von Neumann Theorem, $\cF_B$ is the Fock space representation of $\fh_B \otimes_{\F_{q,\pi}} \F$.  (See, for example, \cite[Th.~2.11]{SY15} for a formulation of the Stone--von Neumann Theorem adapted to lattice Heisenberg algebras and more general Heisenberg doubles.)  Note that we need to tensor over $\F_{q,\pi}$ with $\F$ here since the action of $\Tr(\cH_B)$ is via the injection \eqref{eq:trace-injection}.  The result now follows from Theorem~\ref{theo:trace-identification}.
\end{proof}

\begin{rem}
  If $\delta > 0$, it follows from degree considerations that the action of $\Tr(\cH_B) \cong \fh_B \otimes_{\F_{q,\pi}} \F$ on $Z(\cH_B^*)$ is a lowest-weight representation (i.e.\ it is generated by lowest weight vectors).  Therefore, by the Stone--von Neumann Theorem, knowing the space of lowest-weight vectors would yield an explicit description of $Z(\cH_B^*)$.  A description of $Z(\cH_B^*)$ was conjectured in \cite[Conj.~8.10]{RS15}.
\end{rem}

\begin{prop} \label{prop:wreath-closures-span}
  We have
  \[
    \sum_{n=0}^\infty [(B^{\rtimes n})^\op] \cdot 1_\varnothing = Z(\cH_B^*),
  \]
  where $[(B^{\rtimes n})^\op] \subseteq \Tr(\cH_B^*)^+$ via the first algebra homomorphism of \eqref{eq:FSn-end-spaces},
\end{prop}

\begin{proof}
  By \cite[Prop.~8.9]{RS15}, $Z(\cH_B^*)$ is spanned by products of clockwise circles carrying dots and right curls.  Using \cite[Prop.~8.8]{RS15}, one can nest such circles inside one another.  It follows that $\Tr(\cH_B^*) \cdot 1_\varnothing = Z(\cH_B^*)$.  Then it follows by degree considerations, as in the proof of Proposition~\ref{prop:Fock-space-isom}, that $\Tr(\cH_B^*)^+ \cdot 1_\varnothing = Z(\cH_B^*)$.

  It remains to show that if $z \in \END_{\cH_B}(\sQ_+^n)$ is a diagram consisting of a permutation of the strands, carrying dots and right curls, then we can write $[z] \cdot 1_\varnothing$ as a linear combination of elements of the form $[z'] \cdot 1_\varnothing$, where $z' \in (B^{\rtimes m})^\op$, i.e., where $z'$ does not contain right curls.  We prove this result by induction on the number of right curls in $z$.

  Using \eqref{rel:up-up-double-cross} and \cite[(8.4)]{RS15}, we can move right curls to the right, modulo diagrams with fewer right curls (recall the notation \eqref{eq:right-curl}):
  \[
    \begin{tikzpicture}[anchorbase]
      \draw[->] (0,0) --(0,2);
      \draw[->] (1,0) -- (1,2);
      \redcircle{(0,0.5)};
    \end{tikzpicture}
    \ =\
    \begin{tikzpicture}[anchorbase]
      \draw[->] (0,0) .. controls (1,1) .. (0,2);
      \draw[->] (1,0) .. controls (0,1) .. (1,2);
      \redcircle{(0.25,0.25)};
    \end{tikzpicture}
    \ =\
    \begin{tikzpicture}[anchorbase]
      \draw[->] (0,0) .. controls (1,1) .. (0,2);
      \draw[->] (1,0) .. controls (0,1) .. (1,2);
      \redcircle{(0.73,1)};
    \end{tikzpicture}
    \ + \sum_{b \in \cB}\
    \begin{tikzpicture}[anchorbase]
      \draw[->] (0,0) .. controls (0,1) .. (1,2);
      \draw[->] (1,0) .. controls (1,1) .. (0,2);
      \bluedot{(0,0.25)} node [anchor=east,color=black] {$b$};
      \bluedot{(0.99,0.5)} node [anchor=west,color=black] {$b^\vee$};
    \end{tikzpicture}
  \]
  We can thus choose a right curl in $z$ and move it to the rightmost strand, modulo diagrams with fewer right curls.  Therefore we can assume one of the right curls in $z$ is on the rightmost strand.  But then $[z] \cdot 1_\varnothing = [z'] \cdot 1_\varnothing$, where $z'$ is obtained from $z$ by removing a right curl from the rightmost strand, adding a new strand on the far right, and adding a crossing between the two rightmost strands at the location of the removed right curl:
  
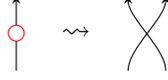
\[
    \begin{tikzpicture}[anchorbase]
      \draw[->] (0,0) -- (0,1);
      \redcircle{(0,0.5)};
    \end{tikzpicture}
    \quad \rightsquigarrow \quad
    \begin{tikzpicture}[anchorbase]
      \draw[->] (0,0) .. controls (0,0.3) and (0.5,0.7) .. (0.5,1);
      \draw[->] (0.5,0) .. controls (0.5,0.3) and (0,0.7) .. (0,1);
    \end{tikzpicture}
  \]
  This completes the proof of the induction step.
\end{proof}

%
\section{A bilinear form on annular diagrams} \label{sec:diagrammatic-pairing}
%

\subsection{The bilinear form}

Recall the functor $\bF_0 \colon \cH_B \to \F\md$ defined in \cite[\S7]{RS15}.  This functor maps the graded center $Z(\cH_B^*)$ of the category $\cH_B$ to the center $Z(\F\md^*)$ of the graded category $\F\md^*$ of $\Z$-graded super vector spaces over $\F$.  Via the natural identification of $Z(\F\md^*)$ with $\F$, we will view $\bF_0$ as giving a map $Z(\cH_B^*) \to \F$.

An equivalent definition of $\bF_0$ is that it is the $\F$-linear projection
\[
  \bF_0 \colon Z(\cH_B^*) \to Z(\cH_B) \cong \F.
\]
of the center of the graded category $\cH_B^*$ onto its degree zero piece, which is the center of the category $\cH_B$.

Let
\begin{equation} \label{eq:ev1}
  \ev_{1,-1} \colon \F_{q,\pi} \to \F
\end{equation}
denote the evaluation at $q=1$ and $\pi=-1$.

\begin{lem} \label{lem:proj0-equivariance}
  The diagram
  \[
    \xymatrix{
      \Tr(\cH_B) \ar[r]^{\alpha} \ar[d]_{\varphi_B} & \cF_B \ar[r]^{\bF_0} \ar[d]_{\phi_B} & \F \\
      \fh_B \ar[r]^{\alpha'} & F_{B,\F} \ar[ur]_{\kappa_{0,\F}} &
    }
  \]
  commutes, where the maps $\alpha$ and $\alpha'$ are given by action on the vacuum vectors $1_\varnothing$ and $1_{\fh_B^+} \in F_B$, respectively, and where $\phi_B$ is as defined in \eqref{eq:Fock-space-isom}, and $\kappa_{0,\F}$ is the map \eqref{eq:pi0-def} followed by the map $\ev_{1-1}$.
\end{lem}

\begin{proof}
  The commutativity of the left-hand square follows from Proposition~\ref{prop:Fock-space-isom}.  To prove commutativity of the right-hand triangle, suppose $\lambda$ is a nonempty partition and $i \in \{1,\dotsc,\ell\}$.  Then it is immediate from the definition in \cite[\S7.1]{RS15} that $\bF_0([e_{\lambda,i}^+])=0$, since the indices as described in \cite[\S7.1]{RS15} become negative.  The result then follows immediately from Proposition~\ref{prop:Fock-space-isom}.
\end{proof}

It follows from Lemma~\ref{lem:proj0-equivariance} that $\varphi_B$ intertwines the linear functionals
\[
  \Tr(\cH_B) \to \F,\quad x \mapsto \bF_0(x \cdot 1_\varnothing),
  \quad \text{and} \quad
  \fh_B \to \F,\quad x \mapsto \kappa_0 \left( x \cdot 1_{\fh_B^+} \right).
\]
Define a pairing
\begin{equation} \label{eq:diagrammatic-pairing}
  \langle -, - \rangle_B \colon \Tr(\cH_B)^- \times \Tr(\cH_B)^+ \to \F,\quad
  \langle x, y \rangle_B = \bF_0 \big( (xy) \cdot 1_\varnothing \big).
\end{equation}
Diagrammatically, the pairing $\langle x, y \rangle_B$ is obtained by juxtaposing the diagrams for $x$ and $y$, closing off the resulting composite diagram to the right, and applying the functor $\bF_0$:
\[
  \langle x,y \rangle_B = \bF_0
  \left(
    \begin{tikzpicture}[anchorbase]
      \draw (-1,-0.25) -- (-1,0.25) -- (-0.5,0.25) -- (-0.5,-0.25) -- (-1,-0.25);
      \draw (-0.4,-0.25) -- (-0.4,0.25) -- (0.3,0.25) -- (0.3,-0.25) -- (-0.4,-0.25);
      \draw (-0.75,0) node {$x$};
      \draw (-0.05,0) node {$y$};
      \draw[<-] (-0.85,0.25) -- (-0.85,0.5) .. controls (-0.85,1.7) and (2.85,1.7) .. (2.85,0.5) -- (2.85,-0.5) .. controls (2.85,-1.7) and (-0.85,-1.7) .. (-0.85,-0.5) -- (-0.85,-0.25);
      \draw[<-] (-0.65,0.25) -- (-0.65,0.5) .. controls (-0.65,1.5) and (2.65,1.5) .. (2.65,0.5) -- (2.65,-0.5) .. controls (2.65,-1.5) and (-0.65,-1.5) .. (-0.65,-0.5) -- (-0.65,-0.25);
      \draw[->] (-0.25,0.25) -- (-0.25,0.5) .. controls (-0.25,1.1) and (2.25,1.1) .. (2.25,0.5) -- (2.25,-0.5) .. controls (2.25,-1.1) and (-0.25,-1.1) .. (-0.25,-0.5) -- (-0.25,-0.25);
      \draw[->] (-0.05,0.25) -- (-0.05,0.5) .. controls (-0.05,0.9) and (2.05,0.9) .. (2.05,0.5) -- (2.05,-0.5) .. controls (2.05,-0.9) and (-0.05,-0.9) .. (-0.05,-0.5) -- (-0.05,-0.25);
      \draw[->] (0.15,0.25) -- (0.15,0.4) .. controls (0.15,0.7) and (1.85,0.7) .. (1.85,0.4) -- (1.85,-0.4) .. controls (1.85,-0.7) and (0.15,-0.7) .. (0.15,-0.4) -- (0.15,-0.25);
    \end{tikzpicture}
  \right)\ .
\]

\begin{rem}
  In fact, the pairing \eqref{eq:diagrammatic-pairing} can be extended to a pairing for the larger trace $\Tr(\cH_B^*)$, using the same definition.  In this setting, it is more natural to consider the 2-category version of the Heisenberg category $\cH_B$ (see \cite[Rem.~6.1]{RS15}), in which case the definition of the trace of a 2-category implies that we only pair elements of $\Tr(\cH_B^*)^+$ of ascension $n$ with elements of $\Tr(\cH_B^*)^-$ of ascension $-n$ (so that their horizontal composition is the endomorphism of some object, as opposed to a morphism between distinct objects).  We will not investigate the properties of this extended bilinear form in the current paper.
\end{rem}

\begin{rem}
  One could modify the pairing \eqref{eq:diagrammatic-pairing} by replacing $\bF_0$ with the functor $\bF_n$ for any $n \in \N$ (see \cite[\S7]{RS15}).  In this case, the pairing would take values in the center $Z(B^{\rtimes n}\md^*)$ of the graded category $B^{\rtimes n}\md^*$, which can be naturally identified with the center $Z(B^{\rtimes n})$ of the wreath product algebra $B^{\rtimes n}$.  It would be interesting to further examine this family of pairings.
\end{rem}

\subsection{Isometries}

Recalling \eqref{eq:ev1}, we let
\[
  \langle -,- \rangle_{1,-1} = \ev_{1,-1} \circ \langle -,- \rangle
  \colon \fh_B^- \otimes \fh_B^+ \to \F
\]
denote the corresponding specialization of the pairing of $\fh_B^\pm$.

\begin{cor}
  The diagram
  \[
    \xymatrix{
      \Tr(\cH_B)^- \times \Tr(\cH_B)^+ \ar[d]_{\varphi_B \times \varphi_B} \ar[rr]^(0.65){\langle -, - \rangle_B} & & \F \\
      \fh_B^- \times \fh_B^+ \ar[urr]_{\langle -, - \rangle_{1,-1}} & &
    }
  \]
  commutes.
\end{cor}

\begin{proof}
  This follows directly from Lemmas~\ref{lem:recover-pairing} and~\ref{lem:proj0-equivariance}.
\end{proof}

For each $n \in \N$, the adjunctions in the category $\cH_B$ give us an automorphism of $\F_{q,\pi}$-algebras
\[
  \End_{\cH_B}(\sQ_+^n)^\op \to \End_{\cH_B}(\sQ_-^n),\quad x \mapsto x^\rot,
\]
given by rotating a diagram through an angle $\pi$.  In this way we can define a bilinear form
\begin{equation} \label{eq:diagrammatic-form}
  \langle -, - \rangle_B \colon \Tr(\cH_B)^+ \times \Tr(\cH_B)^+ \to \F,\quad \langle x, y \rangle_B
  = \langle x^\rot, y \rangle_B
  = \bF_0\big( (x^\rot y) \cdot 1_\varnothing \big).
\end{equation}

\begin{theo} \label{theo:Jack-diagrammatic}
  Suppose $B_0 = \F$, so that $\fh_B^+ = \Sym$.  Then the algebra isomorphism $\varphi_B \colon \Tr(\cH_B)^+ \to \fh_B^+ = \Sym$ is an isometry, where $\Tr(\cH_B)^+$ is equipped with the bilinear form \eqref{eq:diagrammatic-form} and $\Sym$ is equipped with the Jack bilinear form at parameter $\dim B_\even - \dim B_\odd$ (see \eqref{eq:Jack-pairing}).
\end{theo}

\begin{proof}
  Since $B_0 = \F$, the algebra $B$ has only one simple module (up to grading and parity shift), namely the one-dimensional module $B/\bigoplus_{n>0} B_n$.  The projective cover of this simple module is $B$ itself.  Then we have $\HOM_B(B,B) \cong B$, as $\Z_{q,\pi}$-modules.  For $n \in \N$, we have $(P_n^+)^\rot = P_n^-$, and
  \[
    \langle p_n^-, p_m^+ \rangle
    = \delta_{n,m} n \theta_n ( \grdim B )
    \implies \langle p_n^+, p_m^+ \rangle_{1,-1}
    = \delta_{n,m} n (\dim B_\even - \dim B_\odd),
  \]
  which coincides with the Jack bilinear form at parameter $\dim B_\even - \dim B_\odd$.  Since any Hopf pairing on $\Sym$ is uniquely determined by its value on the power sums, the proposition follows.
\end{proof}

\begin{rem}
  If we fix $k \in \N_+$, then the algebra $B = \F[x]/(x^k)$ is naturally a $\Z$-graded Frobenius algebra, where we declare $x$ to be even of $\Z$-degree one.  Since $B$ has only one indecomposable projective module, we have $\Tr(\cH_B)^\pm \cong \Sym$.  The pairing \eqref{eq:diagrammatic-pairing} then corresponds to the Jack pairing described in \eqref{eq:Jack-pairing} at parameter $k$.

  On the other hand, if we let $B$ be the cohomology of a surface of genus $g$, then $B$ is a supercommutative Frobenius algebra with $\dim B_\even = 2$ and $\dim B_\odd = 2g$.  Thus the pairing \eqref{eq:diagrammatic-pairing} corresponds to the Jack pairing described in \eqref{eq:Jack-pairing} at parameter $2 - 2g$, allowing us to realize the Jack pairing at negative parameters.
\end{rem}

\begin{eg}
  Suppose $B = \F[x]/(x^k)$, with trace map given by $\tr_B ( \sum_{j=1}^{k-1} a_j x^j ) = a_{k-1}$.  The first power sum $p_1^+$ corresponds to the clockwise circle $P_1^+$, and $\left( P_1^+ \right)^\rot = P_1^-$.  We compute
  \begin{multline*}
    \langle P_1^-, P_1^+ \rangle_B
    =
    \begin{tikzpicture}[anchorbase]
      \draw [<-] (0,0) arc (-180:180:0.3);
      \draw [->] (-0.3,0) arc (-180:180:0.6);
    \end{tikzpicture}
    \ \stackrel{\eqref{rel:down-up-double-cross}}{=}\
    \begin{tikzpicture}[anchorbase]
      \draw [->] (0,0) arc (-180:180:0.5);
      \draw [<-] (0.5,0) arc (-180:180:0.5);
    \end{tikzpicture}
    \ + \sum_{j=0}^{k-1}\
    \begin{tikzpicture}[anchorbase]
      \draw[<-] (0,-0.3) -- (0,0) arc (180:0:0.6) arc (360:180:0.2) arc (0:180:0.2) -- (0.4,-0.6) arc (180:360:0.2) arc (180:0:0.2) arc (360:180:0.6) -- (0,-0.3);
      \bluedot{(1.2,0)} node[color=black,anchor=west] {$x^{k-j-1}$};
      \bluedot{(0.8,-0.6)};
      \draw (0.63,-0.4) node {$x^j$};
    \end{tikzpicture}
    \\
    \stackrel{\eqref{rel:up-up-double-cross}}{=}\
    \begin{tikzpicture}[anchorbase]
      \draw [->] (0,0) arc (-180:180:0.5);
    \end{tikzpicture}
    \
    \begin{tikzpicture}[anchorbase]
      \draw [<-] (0,0) arc (-180:180:0.5);
    \end{tikzpicture}
    \ + \sum_{j=0}^{k-1}\
    \begin{tikzpicture}[anchorbase]
      \draw [->] (1,0) arc (0:360:0.5);
      \bluedot{(0,0)} node[color=black,anchor=east] {$x^{k-1}$};
    \end{tikzpicture}
    \ \stackrel{\eqref{rel:ccc}}{=} k.
  \end{multline*}
  This agrees with the inner product of $p_1$ with itself in the Jack inner product at Jack parameter $k = \dim B$.
\end{eg}

\begin{rem}
  The Heisenberg algebra $\fh_B$ is a Heisenberg double, whose definition involves a Hopf pairing between two Hopf algebras.  This pairing is thus implicit in the algebra structure of $\fh_B$, which was realized via Grothendieck group categorification in \cite{RS15} and via trace categorification in Theorem~\ref{theo:trace-identification}.  In particular, for the choice $B = \F[x]/(x^k)$, the pairing in $\fh_B$ is the specialization \eqref{eq:specialized-Macdonald-pairing} of the Macdonald pairing.  The inclusion of $\Tr(\cH_B)$ into $\Tr(\cH_B^*)$, which collapses the grading (see Proposition~\ref{prop:trace-injection}), corresponds to the $q \to 1$ limit that yields the Jack pairing from the Macdonald pairing.
\end{rem}

%
\section{Filtrations and associated graded algebras} \label{sec:filtrations}
%

We conclude the paper with a discussion of a natural filtration on the center $Z(\cH_B^*)$ arising from the ascension grading on the trace $\Tr(\cH_B^*)$.

By Proposition~\ref{prop:wreath-closures-span}, we have an $\F$-vector space filtration of $Z(\cH_B^*)$ given by
\begin{equation} \label{eq:filtration}
  0 \subseteq Z_0 \subseteq Z_1 \subseteq Z_2 \subseteq \dotsb,\qquad
  \text{where } Z_n = \sum_{m=0}^n [(B^{\rtimes m})^\op] \cdot 1_\varnothing.
\end{equation}
We then have the associated graded $\F$-vector space
\begin{equation} \label{eq:center-associated-graded}
  \gr Z(\cH_B^*) = \bigoplus_{n=0}^\infty Z_n/Z_{n-1},
\end{equation}
where we adopt the convention that $Z_{-1}=0$.  For $a \in Z(\cH_B^*)$, we let $\gr (a)$ denote its image under the usual $\F$-linear map from $Z(\cH_B^*)$ to $\gr Z(\cH_B^*)$.

The center $Z(\cH_B^*)$ is naturally an $\F$-algebra, with multiplication given by juxtaposition of diagrams.  The following results gives a precise relationship between the multiplicative structure on $\Tr(\cH_B^*)$ and the multiplicative structure on $Z(\cH_B^*)$.  Throughout this section, in order to simplify diagrams, we will draw a single arc to represent multiple strands.

\begin{prop} \label{prop:encircle-versus-juxtapose}
  If $n_1,n_2 \in \N_+$, $z_1 \in A_{n_1}$, and $z_2 \in A_{n_2}$, then
  \[
    \gr \left(
    \begin{tikzpicture}[anchorbase]
      \draw (-1.7,0) -- (-1.7,0.4) -- (-0.7,0.4) -- (-0.7,0) -- (-1.7,0);
      \draw (-1.2,0.2) node {$z_1$};
      \draw[->] (-1.2,0.4) .. controls (-1.2,2) and (1.4,2) .. (1.4,0) .. controls (1.4,-2) and (-1.2,-2) .. (-1.2,0);
      \draw (-0.5,-0.4) -- (-0.5,0) -- (0.5,0) -- (0.5,-0.4) -- (-0.5,-0.4);
      \draw (0,-0.2) node {$z_2$};
      \draw[->] (0,0) .. controls (0,0.8) and (1,0.8) .. (1,-0.2) .. controls (1,-1) and (0,-1) .. (0,-0.4);
    \end{tikzpicture}
    \right) = \gr \left(
    \begin{tikzpicture}[anchorbase]
      \draw (-0.5,-0.2) -- (-0.5,0.2) -- (0.5,0.2) -- (0.5,-0.2) -- (-0.5,-0.2);
      \draw (0,0) node {$z_1$};
      \draw[->] (0,0.2) .. controls (0,1) and (1,1) .. (1,0) .. controls (1,-1) and (0,-1) .. (0,-0.2);
    \end{tikzpicture}
    \
    \begin{tikzpicture}[>=stealth,baseline={([yshift=2ex]current bounding box.center)}]
      \draw (-0.5,-0.2) -- (-0.5,0.2) -- (0.5,0.2) -- (0.5,-0.2) -- (-0.5,-0.2);
      \draw (0,0) node {$z_2$};
      \draw[->] (0,0.2) .. controls (0,1) and (1,1) .. (1,0) .. controls (1,-1) and (0,-1) .. (0,-0.2);
    \end{tikzpicture}
    \right).
  \]
\end{prop}

\begin{proof}
  It suffices to consider $z_2$ of the form $\bc z_3$ for some $\bc \in B^{\otimes n_2}$ and $z_3 \in \F S_{n_2}$.  Using \eqref{rel:dotslide-up-right}, \eqref{rel:dotslide-up-left} and \eqref{rel:up-up-double-cross} we have
  \[
    \begin{tikzpicture}[anchorbase]
      \draw (-1.7,0) -- (-1.7,0.4) -- (-0.7,0.4) -- (-0.7,0) -- (-1.7,0);
      \draw (-1.2,0.2) node {$z_1$};
      \draw[->] (-1.2,0.4) .. controls (-1.2,2) and (1.4,2) .. (1.4,0) .. controls (1.4,-2) and (-1.2,-2) .. (-1.2,0);
      \draw (-0.5,-0.4) -- (-0.5,0) -- (0.5,0) -- (0.5,-0.4) -- (-0.5,-0.4);
      \draw (0,-0.2) node {$z_2$};
      \draw[->] (0,0) .. controls (0,0.8) and (1,0.8) .. (1,-0.2) .. controls (1,-1) and (0,-1) .. (0,-0.4);
    \end{tikzpicture}
    \ =\
    \begin{tikzpicture}[anchorbase]
      \draw (-1.7,0) -- (-1.7,0.4) -- (-0.7,0.4) -- (-0.7,0) -- (-1.7,0);
      \draw (-1.2,0.2) node {$z_1$};
      \draw[->] (-1.2,0.4) .. controls (-1.2,2) and (1.1,2) .. (1.1,0) .. controls (1.1,-2) and (-1.2,-2) .. (-1.2,0);
      \draw (-0.5,-0.4) -- (-0.5,0) -- (0.5,0) -- (0.5,-0.4) -- (-0.5,-0.4);
      \draw (0,-0.2) node {$z_2$};
      \draw[->] (0,0) .. controls (0,1.5) and (2,1.5) .. (2,-0.2) .. controls (2,-1.5) and (0,-1.5) .. (0,-0.4);
    \end{tikzpicture}
    \ = (-1)^{|\bc| |z_1|}\
    \begin{tikzpicture}[anchorbase]
      \draw (-1.7,0) -- (-1.7,0.4) -- (-0.7,0.4) -- (-0.7,0) -- (-1.7,0);
      \draw (-1.2,0.2) node {$z_1$};
      \draw[->] (-1.2,0.4) .. controls (-1.2,2) and (1.1,2) .. (1.1,0) .. controls (1.1,-2) and (-1.2,-2) .. (-1.2,0);
      \draw (-0.5,-0.4) -- (-0.5,0) -- (0.5,0) -- (0.5,-0.4) -- (-0.5,-0.4);
      \draw (0,-0.2) node {$z_3$};
      \draw[->] (0,0) .. controls (0,1.5) and (2,1.5) .. (2,-0.2) .. controls (2,-1.5) and (0,-1.5) .. (0,-0.4);
      \bluedot{(1.5,0.92)} node[anchor=south, color=black] {\bf{c}};
    \end{tikzpicture}
    \ \ .
  \]
  Now, by \cite[Lem.~8.2]{RS15} we can use triple point moves for strands of any possible orientation.  Thus, we can pull all the strands emanating from the box labeled $z_1$ through the box labeled $z_3$:
  \[
    \begin{tikzpicture}[anchorbase]
      \draw (-1.7,0) -- (-1.7,0.4) -- (-0.7,0.4) -- (-0.7,0) -- (-1.7,0);
      \draw (-1.2,0.2) node {$z_1$};
      \draw[->] (-1.2,0.4) .. controls (-1.2,2) and (1.1,2) .. (1.1,0) .. controls (1.1,-2) and (-1.2,-2) .. (-1.2,0);
      \draw (-0.5,-0.4) -- (-0.5,0) -- (0.5,0) -- (0.5,-0.4) -- (-0.5,-0.4);
      \draw (0,-0.2) node {$z_3$};
      \draw[->] (0,0) .. controls (0,1.5) and (2,1.5) .. (2,-0.2) .. controls (2,-1.5) and (0,-1.5) .. (0,-0.4);
      \bluedot{(1.5,0.92)} node[anchor=south, color=black] {\bf{c}};
    \end{tikzpicture}
    \ =\
    \begin{tikzpicture}[anchorbase]
      \draw (-1.7,0.2) -- (-1.7,0.6) -- (-0.7,0.6) -- (-0.7,0.2) -- (-1.7,0.2);
      \draw (-1.2,0.4) node {$z_1$};
      \draw (-1.2,0.6) .. controls (-1.2,1) and (-0.8,1.2) .. (0,1.2) .. controls (0.2,1.2) and (0.7,1) .. (0.7,0.7).. controls (0.7, 0.4) and (0,0.4) .. (0,0.2);
      \draw[<-] (-1.2,0.2) .. controls (-1.2,-0.5) and (0,-0.5) .. (0,0.2);
      \draw (0.5,-0.2) -- (0.5,0.2) -- (1.5,0.2) -- (1.5,-0.2) -- (0.5,-0.2);
      \draw (1,0) node {$z_3$};
      \draw(2.2,0.2) .. controls (2.2,1) and (1.8,1.2) .. (1,1.2) .. controls (0.8,1.2) and (0.3,1) .. (0.3,0.7).. controls (0.3, 0.4) and (1,0.4) .. (1,0.2);
      \draw[->] (2.2,-0.2) .. controls (2.2,-1) and (1,-1) .. (1,-0.2);
      \draw (2.2,-0.2) -- (2.2,0.2);
      \bluedot{(0.9,1.2)} node[anchor=south, color=black] {\bf{c}};
    \end{tikzpicture}
  \]
  We can then use \eqref{rel:down-up-double-cross} to pull apart the double crossings one by one, introducing terms with dots and fewer strands:
  \[
    \begin{tikzpicture}[anchorbase]
      \draw (-1.7,0.2) -- (-1.7,0.6) -- (-0.7,0.6) -- (-0.7,0.2) -- (-1.7,0.2);
      \draw (-1.2,0.4) node {$z_1$};
      \draw (-1.2,0.6) .. controls (-1.2,1) and (-0.8,1.2) .. (0,1.2) .. controls (0.2,1.2) and (0.7,1) .. (0.7,0.7).. controls (0.7, 0.4) and (0,0.4) .. (0,0.2);
      \draw[<-] (-1.2,0.2) .. controls (-1.2,-0.5) and (0,-0.5) .. (0,0.2);
      \draw (0.5,-0.2) -- (0.5,0.2) -- (1.5,0.2) -- (1.5,-0.2) -- (0.5,-0.2);
      \draw (1,0) node {$z_3$};
      \draw(2.2,0.2) .. controls (2.2,1) and (1.8,1.2) .. (1,1.2) .. controls (0.8,1.2) and (0.3,1) .. (0.3,0.7).. controls (0.3, 0.4) and (1,0.4) .. (1,0.2);
      \draw[->] (2.2,-0.2) .. controls (2.2,-1) and (1,-1) .. (1,-0.2);
      \draw (2.2,-0.2) -- (2.2,0.2);
      \bluedot{(0.9,1.2)} node[anchor=south, color=black] {\bf{c}};
    \end{tikzpicture}
    \ \equiv\
    \begin{tikzpicture}[anchorbase]
      \draw (-0.5,-0.2) -- (-0.5,0.2) -- (0.5,0.2) -- (0.5,-0.2) -- (-0.5,-0.2);
      \draw (0,0) node {$z_1$};
      \draw[->] (0,0.2) .. controls (0,1) and (1,1) .. (1,0) .. controls (1,-1) and (0,-1) .. (0,-0.2);
    \end{tikzpicture}
    \
    \begin{tikzpicture}[>=stealth,baseline={([yshift=1ex]current bounding box.center)}]
      \draw (-0.5,-0.2) -- (-0.5,0.2) -- (0.5,0.2) -- (0.5,-0.2) -- (-0.5,-0.2);
      \draw (0,0) node {$z_3$};
      \draw[->] (0,0.2) .. controls (0,1.2) and (1,1.2) .. (1,0) .. controls (1,-1) and (0,-1) .. (0,-0.2);
      \bluedot{(0.5,0.92)} node[anchor=south, color=black] {\bf{c}};
    \end{tikzpicture}
    \ =(-1)^{|\bc||z_1|}\
    \begin{tikzpicture}[anchorbase]
      \draw (-0.5,-0.2) -- (-0.5,0.2) -- (0.5,0.2) -- (0.5,-0.2) -- (-0.5,-0.2);
      \draw (0,0) node {$z_1$};
      \draw[->] (0,0.2) .. controls (0,1) and (1,1) .. (1,0) .. controls (1,-1) and (0,-1) .. (0,-0.2);
    \end{tikzpicture}
    \
    \begin{tikzpicture}[>=stealth,baseline={([yshift=2ex]current bounding box.center)}]
      \draw (-0.5,-0.2) -- (-0.5,0.2) -- (0.5,0.2) -- (0.5,-0.2) -- (-0.5,-0.2);
      \draw (0,0) node {$z_2$};
      \draw[->] (0,0.2) .. controls (0,1) and (1,1) .. (1,0) .. controls (1,-1) and (0,-1) .. (0,-0.2);
    \end{tikzpicture}
  \]
  Here we used the symbol $\equiv$ for diagrams to indicate that they have the same image in the graded vector space.
\end{proof}

\begin{cor} \label{cor:center-filtration}
  The filtration \eqref{eq:filtration} is a filtration of $\F$-algebras and hence \eqref{eq:center-associated-graded} is the associated graded algebra.
\end{cor}

\begin{proof}
  Consider the filtration on $\Tr(\cH_B^*)$ induced by the ascension grading.  Then the action of $\Tr(\cH_B^*)$ on $1_\varnothing$ sends the $n$-th step of the filtration on $\Tr(\cH_B^*)$ to $Z_n$.  The corollary then follows from Proposition~\ref{prop:encircle-versus-juxtapose}.
\end{proof}

\begin{prop} \label{prop:Pn-dotted-circle}
  For $n \in \N_+$ and $b \in B$ we have
  \[
    \gr \left( \left( \left( b \otimes 1_B^{\otimes (n-1)} \right) \circ P_n^+ \right) \cdot 1_\varnothing \right)
    = \gr \left(
    \begin{tikzpicture}[anchorbase]
      \draw[<-] (0,0) arc (90:470:.5);
      \redcircle{(0.45,-0.7)} node[color=black,anchor=west] {$n-1$};
      \bluedot{(-0.45,-0.3)} node[color=black,anchor=east] {$b$};
    \end{tikzpicture}
    \right).
  \]
  (Here $\circ$ denotes the composition in $\END_{\cH_B}(\sQ_+^n)$.)
\end{prop}

\begin{proof}
  For $m \in \N_+$, let $w_m = (m,m-1,\dotsc,2,1) \in S_m$.  In what follows, we use the symbol $\equiv$ for diagrams to indicate that they have the same image in the associated graded vector space.  By \cite[(8.3)]{RS15}, we have
  \begin{multline*}
    \left( \left( b \otimes 1_B^{\otimes (n-1)} \right) \circ P_n^+ \right) \cdot 1_\varnothing
    = \left[
    \begin{tikzpicture}[anchorbase]
      \draw (-1,-.2) -- (-1,.2) -- (1,.2) -- (1,-.2) -- (-1,-.2);
      \draw (0,0) node {$w_n$};
      \draw[->] (-.8,.2) -- (-.8,1);
      \draw[->] (-.5,.2) -- (-.5,1);
      \draw (0,.6) node {$\cdots$};
      \draw[->] (.5,.2) -- (.5,1);
      \draw[->] (.8,.2) -- (.8,1);
      \draw[->] (-.8,-1) -- (-.8,-.2);
      \draw[->] (-.5,-1) -- (-.5,-.2);
      \draw (0,-.6) node {$\cdots$};
      \draw[->] (.5,-1) -- (.5,-.2);
      \draw[->] (.8,-1) -- (.8,-.2);
      \bluedot{(-.8,.7)} node[color=black,anchor=east] {$b$};
    \end{tikzpicture}
    \right] \cdot 1_\varnothing
    = \left[
    \begin{tikzpicture}[anchorbase]
      \draw (-1,-.2) -- (-1,.2) -- (1,.2) -- (1,-.2) -- (-1,-.2);
      \draw (0,0) node {$w_{n-1}$};
      \draw[->] (-.8,.2) -- (-.8,1);
      \draw[->] (-.5,.2) -- (-.5,1);
      \draw (0,.6) node {$\cdots$};
      \draw[->] (.5,.2) -- (.5,1);
      \draw[->] (.8,.2) -- (.8,1);
      \draw[->] (-.8,-1) -- (-.8,-.2);
      \draw[->] (-.5,-1) -- (-.5,-.2);
      \draw (0,-.6) node {$\cdots$};
      \draw[->] (.5,-1) -- (.5,-.2);
      \draw[->] (.8,-1) -- (.8,-.2);
      \redcircle{(.8,-.6)};
      \bluedot{(-.8,.7)} node[color=black,anchor=east] {$b$};
    \end{tikzpicture}
    \right] \cdot 1_\varnothing
    \\
    = \left[
    \begin{tikzpicture}[anchorbase]
      \draw (-1,-.2) -- (-1,.2) -- (1,.2) -- (1,-.2) -- (-1,-.2);
      \draw (0,0) node {$w_{n-1}$};
      \draw[->] (-.8,.2) -- (-.8,1);
      \draw[->] (-.5,.2) -- (-.5,1);
      \draw (0,.6) node {$\cdots$};
      \draw[->] (.5,.2) -- (.5,1);
      \draw[->] (.8,.2) -- (.8,1);
      \draw[->] (-.8,-1) -- (-.8,-.2);
      \draw[->] (-.5,-1) -- (-.5,-.2);
      \draw (0,-.6) node {$\cdots$};
      \draw[->] (.5,-1) -- (.5,-.2);
      \draw[->] (.8,-1) -- (.8,-.2);
      \redcircle{(.8,.6)};
      \bluedot{(-.8,.7)} node[color=black,anchor=east] {$b$};
    \end{tikzpicture}
    \right] \cdot 1_\varnothing
    \ -\
    \sum_{c \in \cB} \left[
    \begin{tikzpicture}[anchorbase]
      \draw (-1,-.2) -- (-1,.2) -- (1,.2) -- (1,-.2) -- (-1,-.2);
      \draw (0,0) node {$w_{n-2}$};
      \draw[->] (-.8,.2) -- (-.8,1);
      \draw[->] (-.5,.2) -- (-.5,1);
      \draw (0,.6) node {$\cdots$};
      \draw[->] (.5,.2) -- (.5,1);
      \draw[->] (.8,.2) -- (.8,1);
      \draw[->] (-.8,-1) -- (-.8,-.2);
      \draw[->] (-.5,-1) -- (-.5,-.2);
      \draw (0,-.6) node {$\cdots$};
      \draw[->] (.5,-1) -- (.5,-.2);
      \draw[->] (.8,-1) -- (.8,-.2);
      \draw[->] (1.5,-1) -- (1.5,1);
      \bluedot{(.8,.45)} node[anchor=west,color=black] {$c^\vee$};
      \bluedot{(1.5,-.5)} node[anchor=west,color=black] {$c$};
      \bluedot{(-.8,.7)} node[color=black,anchor=east] {$b$};
    \end{tikzpicture}
    \right] \cdot 1_\varnothing
    \equiv \left[
    \begin{tikzpicture}[anchorbase]
      \draw (-1,-.2) -- (-1,.2) -- (1,.2) -- (1,-.2) -- (-1,-.2);
      \draw (0,0) node {$w_{n-1}$};
      \draw[->] (-.8,.2) -- (-.8,1);
      \draw[->] (-.5,.2) -- (-.5,1);
      \draw (0,.6) node {$\cdots$};
      \draw[->] (.5,.2) -- (.5,1);
      \draw[->] (.8,.2) -- (.8,1);
      \draw[->] (-.8,-1) -- (-.8,-.2);
      \draw[->] (-.5,-1) -- (-.5,-.2);
      \draw (0,-.6) node {$\cdots$};
      \draw[->] (.5,-1) -- (.5,-.2);
      \draw[->] (.8,-1) -- (.8,-.2);
      \redcircle{(.8,.6)};
      \bluedot{(-.8,.7)} node[color=black,anchor=east] {$b$};
    \end{tikzpicture}
    \right] \cdot 1_\varnothing
    \\
    = \left[
    \begin{tikzpicture}[anchorbase]
      \draw (-1,-.2) -- (-1,.2) -- (1,.2) -- (1,-.2) -- (-1,-.2);
      \draw (0,0) node {$w_{n-2}$};
      \draw[->] (-.8,.2) -- (-.8,1);
      \draw[->] (-.5,.2) -- (-.5,1);
      \draw (0,.6) node {$\cdots$};
      \draw[->] (.5,.2) -- (.5,1);
      \draw[->] (.8,.2) -- (.8,1);
      \draw[->] (-.8,-1) -- (-.8,-.2);
      \draw[->] (-.5,-1) -- (-.5,-.2);
      \draw (0,-.6) node {$\cdots$};
      \draw[->] (.5,-1) -- (.5,-.2);
      \draw[->] (.8,-1) -- (.8,-.2);
      \redcircle{(.8,.6)};
      \redcircle{(.8,-.6)};
      \bluedot{(-.8,.7)} node[color=black,anchor=east] {$b$};
    \end{tikzpicture}
    \right] \cdot 1_\varnothing
    \equiv \left[
    \begin{tikzpicture}[anchorbase]
      \draw (-1,-.2) -- (-1,.2) -- (1,.2) -- (1,-.2) -- (-1,-.2);
      \draw (0,0) node {$w_{n-2}$};
      \draw[->] (-.8,.2) -- (-.8,1);
      \draw[->] (-.5,.2) -- (-.5,1);
      \draw (0,.6) node {$\cdots$};
      \draw[->] (.5,.2) -- (.5,1);
      \draw[->] (.8,.2) -- (.8,1);
      \draw[->] (-.8,-1) -- (-.8,-.2);
      \draw[->] (-.5,-1) -- (-.5,-.2);
      \draw (0,-.6) node {$\cdots$};
      \draw[->] (.5,-1) -- (.5,-.2);
      \draw[->] (.8,-1) -- (.8,-.2);
      \redcircle{(.8,.6)} node[anchor=west,color=black] {$2$};
      \bluedot{(-.8,.7)} node[color=black,anchor=east] {$b$};
    \end{tikzpicture}
    \right] \cdot 1_\varnothing
    \equiv \dotsb \equiv
    \left[
    \begin{tikzpicture}[anchorbase]
      \draw[->] (0,-1) -- (0,1);
      \redcircle{(0,0)} node[anchor=west,color=black] {$n-1$};
      \bluedot{(0,.7)} node[color=black,anchor=east] {$b$};
    \end{tikzpicture}
    \right] \cdot 1_\varnothing
  \end{multline*}
\end{proof}

\begin{rem}
  The clockwise circle appearing in the statement of Proposition~\ref{prop:Pn-dotted-circle} is denoted $c_{b,{n-1}}$ in \cite[(8.9)]{RS15}.  It was shown in \cite[Prop.~8.9]{RS15} that such diagrams generate $Z(\cH_B^*)$ as an algebra, a result that also follows from Propositions~\ref{prop:wreath-closures-span} and~\ref{prop:Pn-dotted-circle}.
\end{rem}

As a corollary of Propositions~\ref{prop:encircle-versus-juxtapose} and~\ref{prop:Pn-dotted-circle}, we obtain a basis of the center of $Z(\cH_\F)$.  See \cite{KLM16} for a different proof of this statement.

\begin{cor} \label{cor:P-lambda-basis}
  If $B = \F$, the elements $P_\lambda^+ \cdot 1_\varnothing$ form a basis of $Z(\cH_\F)$.
\end{cor}

\begin{proof}
  It was shown in \cite[Prop.~3]{Kho14} that $Z(\cH_B)$ is isomorphic to the polynomial algebra $\F[c_0,c_1,c_2,\dotsc]$, where $c_k$ is a clockwise circle with $k$ right curls.  By Proposition~\ref{prop:Pn-dotted-circle} and Proposition~\ref{prop:encircle-versus-juxtapose}, we have
  \[
    \gr(P_\lambda^+ \cdot 1) = \gr(c_{\lambda_1} c_{\lambda_2} \dotsm c_{\lambda_\ell}),\quad
    \lambda = (\lambda_1,\dotsc,\lambda_\ell) \in \partition.
  \]
  The result follows.
\end{proof}

\begin{rem}
  In the case $B = \F$, the main result of \cite{KLM16} identifies $Z(\cH_\F)$ with the algebra $\Sym'$ of shifted symmetric functions defined by Okounkov and Olshanski in \cite{OO97}.  The algebra $\Sym'$ is a filtered algebra whose associated graded algebra is the algebra $\Sym$ of symmetric functions.  Moreover, in the case $B=\F$, the diagrams in the statement of Proposition~\ref{prop:Pn-dotted-circle} inherit an explicit combinatorial interpretation in the language of shifted symmetric functions.  In particular, the isomorphism
  \[
    Z(\cH_\F)\cong \Sym',
  \]
  is an isomorphism of filtered algebras, and the diagram $P_n^+ \cdot 1_\varnothing$ is mapped to the shifted power sum denoted $p_n^{\#}$ in \cite{OO97};
  the diagram
  
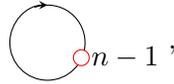
\[
    \begin{tikzpicture}[anchorbase]
      \draw[<-] (0,0) arc (90:470:.5);
      \redcircle{(0.45,-0.7)} node[color=black,anchor=west] {$n-1$};
    \end{tikzpicture},
  \]
  on the other hand, is mapped to the $n$th Boolean cumulant of Kerov's transition measure.  (We refer to \cite{KLM16} for the precise definitions of both shifted symmetric functions and of Kerov's transition measure).  Proposition~\ref{prop:Pn-dotted-circle} then corresponds to the fact that the shifted power sum and the Boolean cumulant are both deformations of the power sum symmetric function.
\end{rem}

%
\section{Further directions} \label{sec:further-directions}
%

The current paper naturally suggests several interesting directions for future research.  We mention some of these here.

\subsection*{More general Frobenius algebras}  We have made several simplifying assumptions in the current paper.  In particular, while the Heisenberg categories $\cH_B$ defined in \cite{RS15} are valid for an arbitrary graded Frobenius superalgebra $B$, we assume in the current paper that the trace map of $B$ is supersymmetric and even.  (In particular, the Nakayama automorphism of $B$ is trivial.)  We also assume that all simple $B$-modules are of type $M$ (i.e.\ not isomorphic to their parity shifts).  See Section~\ref{subsec:hB}.  Allowing $B$ to have simple modules of type $Q$ would result in the appearance of the space of Schur $Q$-functions.  On the other hand, allowing the Nakayama automorphism to be nontrivial (e.g.\ equal to the parity involution) would introduce twisted Heisenberg algebras into the picture (see \cite{CS15,HS16} and \cite[Rem.~6.2]{RS15}).  It would be interesting to pursue these generalizations.

\medskip

\subsection*{Connections to W-algebras} \label{subsec:future-w-algebras}  It is shown in \cite[Th.~1]{CLLS15} that $\Tr(\cH_\F)$ is isomorphic to a quotient of the W-algebra $W_{1+\infty}$.  The action of $\Tr(\cH_\F)$ on $Z(\cH_\F)$ then gives a graphical interpretation of the action of $W_{1+\infty}$ on the space of shifted symmetric functions, which is identified with a level one irreducible representation of $W_{1+\infty}$.  In the more general setting, we expect that $\Tr(\cH_B^*)$ should be related to W-algebras associated to the lattice $K_0(B)$, as conjectured in \cite{CLLS15}.

\subsection*{Wreath product algebras} As described in \cite[\S7]{RS15}, the categories $\cH_B^*$ act on categories of modules over wreath product algebras $B^{\otimes n} \rtimes S_n$, $n \in \N$.  Combined with the results of the current paper, this yields actions of Heisenberg algebras on the centers of these module categories.  One should thus be able to use the diagrammatics of the categories $\cH_B^*$ to study these centers.  For example, one should be able to develop a graphical calculus for centers of wreath product algebras in terms of closed diagrams in the categories $\cH_B^*$.

\subsection*{Jack symmetric functions} In the current paper, we develop a categorification of the inner product used to define the Jack symmetric functions (see Theorem~\ref{theo:Jack-diagrammatic}).  In would interesting to try to give a graphical description of the Jack symmetric functions themselves.  That is, one would like to describe natural annular diagrams in $\Tr(\cH_B)$ that correspond to these functions.

\subsection*{Hilbert schemes} \label{subsubsec:Hilbert-schemes} There is a well known relationship, due to Haiman \cite{Hai01}, between the $\C^* \times \C^*$-equivariant $K$-theory of the Hilbert scheme $\textup{Hilb}(\C^2)$ of points on $\C^2$ and the Macdonald ring of symmetric functions, which realises the basis of Macdonald polynomials as classes of $\C^* \times \C^*$-fixed points in equivariant $K$-theory.  On the other hand, there is a parallel description of the Jack symmetric functions using equivariant homology, due to Nakajima \cite{Nak96} and Li--Qin--Wang \cite{LQW04}.  The relationship between $K$-theory and homology is analogous to the relationship between the Grothendieck group and trace in the current paper.  It would be interesting to directly connect the constructions of the present paper to the works \cite{Nak96,LQW04}, by, for example, constructing a categorical Heisenberg action on the equivariant derived category of $\textup{Hilb}(\C^2)$.

%
\appendix
\section{Presentations of lattice Heisenberg algebras and the Macdonald pairing} \label{appendix}
%

In this appendix, we first deduce presentations of lattice Heisenberg algebras that appear naturally in categorification.  These presentations are used in the proof of Proposition~\ref{prop:hB=hV}.  In particular, we find presentations of lattice Heisenberg algebras arising from the Macdonald inner product that may be of independent interest.  We then explain how the limiting procedure that produces the Jack inner product from the Macdonald inner product (see Section~\ref{subsec:Jack-pairing}) can be interpreted in terms of lattice Heisenberg algebras.

\subsection{Presentations}

Let us recall some facts about various generating sets for $\Sy$.  The generating functions for the elementary and complete symmetric functions are
\[
  E(z) = \sum_{r \ge 0} e_r z^t = \prod_{i \ge 1} (1 + x_i z),\quad
  H(z) = \sum_{r \ge 0} h_r z^r = \prod_{i \ge 1} (1 - x_i z)^{-1}.
\]
The generating function for the power sums is
\[
  P(z) = \sum_{r \ge 1} p_r z^{r-1} = \sum_{i \ge 1} \frac{x_i}{1-x_i z} = \sum_{i \ge 1} \frac{d}{dz} \log \frac{1}{1-x_i z}.
\]
Thus
\begin{equation} \label{PH-relation}
  P(z) = \frac{d}{dz} \log H(z) = H'(z)/H(z).
\end{equation}
Similarly,
\begin{equation} \label{PE-relation}
  P(-z) = \frac{d}{dz} \log E(z) = E'(z)/E(z).
\end{equation}

For a finite-dimensional $\Z$-graded super vector space $V$, we define its graded dimension to be
\[
  \grdim V := \sum_{s \in \Z,\, \epsilon \in \Z_2} q^s \pi^\epsilon \dim V_{s,\epsilon} \in \Z_{q,\pi}.
\]

\begin{prop} \label{prop:hh-presentation}
  The algebra $\fh_L$ is generated by the complete symmetric functions $h_{n,i}^\pm$, $n \in \N_+$, $i \in I$, with relations
  \[
    [h_{n,i}^+, h_{m,j}^+] = 0,\quad
    [h_{n,i}^-,h_{m,j}^-] = 0,\quad
    h_{n,i}^+ h_{m,j}^- = \sum_{\ell=0}^{\min(n,m)} \grdim S^\ell(V) h_{m-\ell,j}^- h_{n-\ell,i}^+,\quad n,m \in \N_+,\ i,j \in I,
  \]
  where $V$ is a $\Z$-graded super vector space with $\grdim V = \langle i, j \rangle$.
\end{prop}

\begin{proof}
  By the general theory of Heisenberg doubles, it suffices to compute the commutation relations between the $h_{n,i}^-$ and $h_{m,j}^+$.  By \eqref{PH-relation}, we have
  \[
    H(z) = \exp \sum_{n \ge 1} \frac{p_n}{n} z^n.
  \]
  Now,
  \begin{multline*}
    \left[ \sum_{r \ge 1} \frac{p_{r,i}^+}{r} z^r, \sum_{k \ge 1} \frac{p_{k,j}^-}{k} w^k \right]
    = \sum_{r \ge 1} \frac{\theta_r(\grdim V)}{r} z^r w^r \\
    = \sum_{s,\epsilon} (-1)^{\epsilon}(\dim V_{s,\epsilon}) \sum_{r \ge 1} \frac{((-\pi)^\epsilon q^s z w)^r}{r}
    = \sum_{s,\epsilon} (-1)^{\epsilon+1} (\dim V_{s,\epsilon}) \log (1 - (-\pi)^\epsilon q^s z w).
  \end{multline*}

  Therefore,
  \begin{align*}
    \sum_{n,m \ge 0} h_{n,i}^+ h_{m,j}^- z^n w^m
    &= H^+_i(z) H^-_j(w) \\
    &= \exp \left[ \sum_{r \ge 1} \frac{p_{r,i}^+}{r} z^r, \sum_{k \ge 1} \frac{p_{k,j}^-}{k} w^k \right] H^-_j(w) H^+_i(z) \\
    &= \left( \prod_{s \in \Z} \frac{(1 + \pi q^s zw)^{\dim V_{s,1}}}{(1-q^s z w)^{\dim V_{s,0}}} \right) \sum_{n,m \ge 0} h_{m,j}^- h_{n,i}^+ z^n w^m  \\
    &= \left( \sum_{\ell \ge 0} z^\ell w^\ell \grdim S^\ell(V) \right) \sum_{n,m \ge 0} h_{m,j}^- h_{n,i}^+ z^n w^m,  \\
  \end{align*}
  and the result follows by equating coefficients.
  \details{In the second equality above, we use that $\exp(B) \exp(A) = \exp([B,A]) \exp(A) \exp(B)$ whenever $[A,B]$ commutes with both $A$ and $B$.}
\end{proof}

\begin{prop} \label{prop:he-presentation}
  The algebra $\fh_L$ is generated by the complete symmetric functions $h_{n,i}^+$ and elementary symmetric functions $e_{n,i}^-$, $n \in \N_+$, $i \in I$, with relations
  \[
    [e_{n,i}^-, e_{m,j}^-] = 0,\quad
    [h_{n,i}^+,h_{m,j}^+],\quad
    h_{n,i}^+ e_{m,j}^- = \sum_{\ell=0}^{\min(n,m)} \grdim \Lambda^\ell(V) e_{m-\ell,j}^- h_{n-\ell,i}^+,\quad n,m \in \N_+,\ i,j \in I,
  \]
  where $V$ is a $\Z$-graded super vector space with $\grdim V = \langle i, j \rangle$.
\end{prop}

\begin{proof}
  By the general theory of Heisenberg doubles, it suffices to compute the commutation relations between the $h_n^+$ and $e_m^-$.  By \eqref{PE-relation}, we have
  \[
    E(z) = \exp \sum_{n \ge 1} (-1)^{n-1} \frac{p_n}{n} z^n.
  \]
  Now,
  \begin{multline*}
    \left[ \sum_{r \ge 1} (-1)^{r-1} \frac{p_{r,i}^+}{r} z^r, \sum_{k \ge 1} \frac{p_{k,j}^-}{k} w^k \right]
    = \sum_{r \ge 1} (-1)^{r-1} \frac{\theta_r(\grdim V)}{r} z^r w^r \\
    = -\sum_{s,\epsilon} (-1)^{\epsilon}(\dim V_{s,\epsilon}) \sum_{r \ge 1} \frac{(-(-\pi)^\epsilon q^s z w)^r}{r}
    = \sum_{s,\epsilon} (-1)^\epsilon (\dim V_{s,\epsilon}) \log (1 + (-\pi)^\epsilon q^s z w).
  \end{multline*}

  Therefore,
  \begin{multline*}
    \sum_{n,m \ge 0} h_{n,i}^+ e_{m,j}^- z^n w^m
    = H^+_i(z) E^-_j(w)
    = \exp \left[ \sum_{r \ge 1} (-1)^{r-1} \frac{p_{r,i}^+}{r} z^r, \sum_{k \ge 1} \frac{p_{k,j}^-}{k} w^k \right] E^-_j(w) H^+_i(z) \\
    = \left( \prod_{s \in \Z} \frac{(1 + q^s zw)^{\dim V_{s,0}}}{(1-\pi q^s z w)^{\dim V_{s,1}}} \right) \sum_{n,m \ge 0} e_{m,j}^- h_{n,i}^+ z^n w^m
    = \left( \sum_{\ell \ge 0} z^\ell w^\ell \grdim \Lambda^\ell(V) \right) \sum_{n,m \ge 0} e_m^- h_n^+ z^n w^m,
  \end{multline*}
  and the result follows by equating coefficients.
  \details{In the second equality above, we use that $\exp(B) \exp(A) = \exp([B,A]) \exp(A) \exp(B)$ whenever $[A,B]$ commutes with both $A$ and $B$.}
\end{proof}

\begin{rem}
  By applying the Hopf automorphism of $\Sy$ that interchanges $e_n$ and $h_n$, we can also obtain presentations in terms of $e_{n,i}^\pm$ and in terms of $h_{n,i}^-$ and $e_{n,i}^+$.
\end{rem}

\begin{rem}
  In the case $L = \Z$, with multiplication as the lattice pairing, Propositions~\ref{prop:hh-presentation} and~\ref{prop:he-presentation} are well known.  For certain other lattices, these presentations also occur in \cite[\S2.2]{CL12} and \cite[\S5]{Ber15}.   The presentations of these propositions also appear in \cite[Prop.~5.5]{RS15}.  However, in that paper there are no $p_n$ and so the connection to the power sum presentation is not given.  In the ungraded non-super case, the maps $\theta_n$ are trivial, so the situation simplifies greatly.  In this setting, Proposition~\ref{prop:hh-presentation} essentially appears in \cite[Lem.~1.2]{Kru15}.
\end{rem}

\subsection{More general gradings}

The lattice Heisenberg algebra construction of Section~\ref{sec:lattice-Heis} can be generalized somewhat.  In particular, for $r \in \N_+$, we can replace $\Z_{q,\pi}$ everywhere by the algebra
\[
  \hat \Z_{q_1,\dotsc,q_r,\pi}
  \subseteq
  \Z \llbracket q_1^{\pm 1},\dotsc, q_r^{\pm 1} \rrbracket [\pi]/(\pi^2-1)
\]
consisting of Laurent series
\[
  \sum_{(s_1,\dotsc,s_r,\epsilon) \in \Z^r \otimes \Z_2} a_{s_1,\dotsc,s_r,\epsilon} q_1^{s_1} \dotsm q_r^{s_r} \pi^\epsilon,\quad a_{s_1,\dotsc,s_r,\epsilon} \in \Z,
\]
such that there exists $N \in \Z$ satisfying
\[
  s_1,\dotsc,s_r < N \implies a_{s_1,\dotsc,s_r,\epsilon} = 0.
\]
We define $\hat \F_{q_1,\dotsc,q_s,\epsilon}$ similarly, and we replace $\theta_n$, $n \in \N_+$, by the $\F$-algebra homomorphism
\[
  \theta_n \colon \hat \F_{q_1,\dotsc,q_s,\epsilon} \to \hat \F_{q_1,\dotsc,q_s,\epsilon},\quad
  \theta_n(q_i) = q_i^n,\quad \theta_n(\pi) = -(-\pi)^n,\quad i \in \{1,\dotsc,r\}.
\]

Remark~\ref{rem:rank-one-pairing}, Proposition~\ref{prop:hh-presentation}, and Proposition~\ref{prop:he-presentation} continue to hold in this more general setting, where the graded vector space $V$ appearing there is now $\Z^r \times \Z_2$-graded and we require that its graded dimension
\[
  \grdim V := \sum_{s_1,\dotsc,s_r \in \Z,\, \epsilon \in \Z_2} q_1^{s_1} \dotsm q_r^{s_r} \pi^\epsilon \dim V_{s_1,\dotsc,s_r,\epsilon}
\]
lie in $\hat \Z_{q_1,\dotsc,q_r,\pi}$.

\subsection{The Macdonald inner product} \label{subsec:Macdonald-pairing}

Consider the symmetric superalgebra
\[
  S(x,y) := T(\C x \oplus \C y)/\langle ab - (-1)^{\bar a \bar b} ba \mid a,b \in \{x,y\} \rangle.
\]
where $T(\C x \oplus \C y)$ denotes the tensor superalgebra on the super vector space spanned by $x$ and $y$, and angled brackets denote the ideal generated by the set they enclose.  If we declare $\deg x = (1,0,0) \in \Z^2 \times \Z_2$ $\deg y = (0,1,1) \in \Z^2 \times \Z_2$, then
\begin{equation} \label{eq:Sxy-Macdonald-grdim}
  \grdim S(x,y)
  = \frac{1+\pi q_2}{1-q_1}
  \in \hat \Z_{q_1,q_2,\pi}.
\end{equation}
Therefore, the pairing~\eqref{eq:full-pairing-determined-by-V} becomes
\begin{equation} \label{eq:Macdonald-pairing}
  \langle p_\lambda^-, p_\mu^+ \rangle = \delta_{\lambda,\mu} z_\lambda \prod_{i=1}^{\ell(\lambda)} \frac{1+\pi q_2^{\lambda_i}}{1 - q_1^{\lambda_i}},\quad \lambda,\mu \in \partition.
\end{equation}
If we specialize to $\pi=-1$, this is precisely the Macdonald pairing, used to define the Macdonald symmetric functions.  See for example, \cite[(VI.1.5)]{Mac95}, where $q_1$ and $q_2$ are denoted $t$ and $q$, respectively.

If we fix a rank one lattice $L$ generated by $v$ with $\langle v, v \rangle = \grdim S(x,y)$, then Propositions~\ref{prop:hh-presentation} and~\ref{prop:he-presentation} give presentations of $\fh_L$ in terms of the graded dimensions of the symmetric and exterior powers of $S(x,y)$.  These graded dimensions can be computed explicitly, as shown by the following two propositions.

\begin{prop}
  If $V$ is a $\Z \times \Z \times \Z_2$-graded vector space with graded dimension $\grdim V = \frac{1+\pi q_2}{1-q_1}$, then
  \[
    \grdim S^k(V) = \frac{(1 + \pi q_2q_1^{k-1})(1 + \pi q_2q_1^{k-2}) \dotsm (1 + \pi q_2q_1)(1 + \pi q_2)}{(1 - q_1^k)(1 - q_1^{k-1}) \dotsm (1-q_1)}.
  \]
\end{prop}

\begin{proof}
  Since
  \[
    \sum_{k \ge 0} z^k \grdim S^k(V) = \prod_{n \ge 0} (1 + \pi q_2 q_1^nz) \prod_{n \ge 0} \frac{1}{1-q_1^nz},
  \]
  it suffices to show that
  \[
    \sum_{k \ge 0} \frac{(1 + \pi q_2q_1^{k-1})(1 + \pi q_2q_1^{k-2}) \dotsm (1 + \pi q_2q_1)(1 + \pi q_2)}{(1 - q_1^k)(1 - q_1^{k-1}) \dotsm (1-q_1)} z^k = \prod_{n \ge 0} (1 + \pi q_2 q_1^nz) \prod_{n \ge 0} \frac{1}{1-q_1^nz}.
  \]

  Define $c_k$, $k \ge 0$, by
  \[
    \sum_{k \ge 0} c_k z^k = \prod_{n \ge 0} (1 + \pi q_2 q_1^nz) \prod_{n \ge 0} \frac{1}{1-q_1^nz}.
  \]
  Clearly $c_0=0$.  Now,
  \begin{align*}
    \sum_{k \ge 0} c_k z^k
    &= \prod_{n \ge 0} (1 + \pi q_2 q_1^nz) \prod_{n \ge 0} \frac{1}{1-q_1^nz} \\
    &= \frac{1 + \pi q_2 z}{1-z} \prod_{n \ge 1} (1 + \pi q_2 q_1^nz) \prod_{n \ge 1} \frac{1}{1-q_1^nz} \\
    &= \frac{1 + \pi q_2 z}{1-z} \prod_{n \ge 0} (1 + \pi q_2 q_1^n(q_1z)) \prod_{n \ge 1} \frac{1}{1-q_1^n(q_1z)} \\
    &= \frac{1 + \pi q_2 z}{1-z} \sum_{k \ge 0} c_k (q_1z)^k.
  \end{align*}
  Thus
  \[
    (1-z) \sum_{k \ge 0} c_k z^k = (1 + \pi q_2z) \sum_{k \ge 0} c_k q_1^k z^k.
  \]
  Therefore, for $k \ge 1$, we have
  \[
    c_k - c_{k-1} = q_1^k c_k + \pi q_2 q_1^{k-1} c_{k-1}
    \implies c_k = \frac{1 + \pi q_2 q_1^{k-1}}{1-q_1^k} c_{k-1}.
  \]
  The result follows.
\end{proof}

\begin{prop}
  If $V$ is a $\Z \times \Z \times \Z_2$-graded vector space with graded dimension $\grdim V = \frac{1+\pi q_2}{1-q_1}$, then
  \[
    \grdim \Lambda^k(V) = \frac{(\pi q_2 + q_1^{k-1}) (\pi q_2 + q_1^{(k-2)}) \dotsm (\pi q_2 + 1)}{(1-q_1^k) (1-q_1^{k-1}) \dotsm (1-q_1)}.
  \]
\end{prop}

\begin{proof}
  Since
  \[
    \sum_{k \ge 0} z^k \grdim S^k(V) = \prod_{n \ge 0} \frac{1}{1 - \pi q_2 q_1^nz} \prod_{n \ge 0} (1+q_1^nz),
  \]
  it suffices to show that
  \[
    \sum_{k \ge 0} \frac{(\pi q_2 + q_1^{k-1}) (\pi q_2 + q_1^{(k-2)}) \dotsm (\pi q_2 + 1)}{(1-q_1^k) (1-q_1^{k-1}) \dotsm (1-q_1)} z^k = \prod_{n \ge 0} \frac{1}{1 - \pi q_2 q_1^nz} \prod_{n \ge 0} (1+q_1^nz).
  \]

  Define $c_k$, $k \ge 0$, by
  \[
    \sum_{k \ge 0} c_k z^k = \prod_{n \ge 0} \frac{1}{1 - \pi q_2 q_1^nz} \prod_{n \ge 0} (1+q_1^nz).
  \]
  Clearly $c_0=0$.  Now,
  \begin{align*}
    \sum_{k \ge 0} c_k z^k
    &= \prod_{n \ge 0} \frac{1}{1 - \pi q_2 q_1^nz} \prod_{n \ge 0} (1+q_1^nz) \\
    &= \frac{1 + z}{1 - \pi q_2 z} \prod_{n \ge 1} \frac{1}{1 - \pi q_2 q_1^nz} \prod_{n \ge 0} (1+q_1^nz) \\
    &= \frac{1 + z}{1 - \pi q_2 z} \prod_{n \ge 0} \frac{1}{1 - \pi q_2 q_1^n(q_1z)} \prod_{n \ge 0} (1+q_1^n(q_1z)) \\
    &= \frac{1 + z}{1 - \pi q_2 z} \sum_{k \ge 0} c_k (q_1z)^k.
  \end{align*}
  Thus
  \[
    (1- \pi q_2z) \sum_{k \ge 0} c_k z^k = (1 + z) \sum_{k \ge 0} c_k q_1^k z^k.
  \]
  It follows that, for $k \ge 1$, we have
  \[
    c_k - \pi q_2 c_{k-1} = q_1^k c_k + q_1^{k-1} c_{k-1} \implies c_k = \frac{\pi q_2 + q_1^{k-1}}{1-q_1^k} c_{k-1}.
  \]
  The result follows.
\end{proof}

\subsection{Passing to the Jack limit} \label{subsec:Jack-limit-differential}

Setting $q = q_2 = q_1^k$ and $\pi=-1$ in \eqref{eq:Macdonald-pairing} yields the pairing \eqref{eq:specialized-Macdonald-pairing}.  The relationship between the choice $V=\C[x]/(x^k)$ of Section~\ref{subsec:Jack-pairing} and the $S(x,y)$ of Section~\ref{subsec:Macdonald-pairing} is as follows.  For $k \geq 0$, define a differential $\partial_k \colon S(x,y) \to S(x,y)$ by setting
\[
	\partial_k(y) = x^k,\quad \partial_k(x) = 0,
\]
and extending $\partial_k$ to the rest of $S(x,y)$ via the Leibnitz rule
$\partial_k(fg)= \partial_k(f)g + (-1)^{|f|}f \partial_k(g)$.  The differential $\partial_k$ endows $S(x,y)$ with the structure of a dg-algebra, which is quasi-isomorphic to its cohomology
\[
	H^*(S(x,y),\partial_k) \cong \C[x]/(x^k),
\]
where $\C[x]/(x^k)$ is regarded as a dg-algebra with zero differential.

In light of the above, one may view the process of ``turning on" the differential $\partial_k$ as a categorification of the process of specialising $q_2 = q_1^k$.


\bibliographystyle{alpha}
\bibliography{LicataRossoSavage}

\end{document}